\documentclass{amsart}

\input{xy}
\xyoption{all}
\usepackage{amsmath}
\usepackage{amssymb}
\usepackage{amscd}
\usepackage{graphicx}
\usepackage{hyperref}
\input{quivers.sty}

\newtheorem{theorem}{Theorem}[section]
\newtheorem{lemma}[theorem]{Lemma}
\newtheorem{corollary}[theorem]{Corollary}
\newtheorem{proposition}[theorem]{Proposition}

\theoremstyle{definition}
\newtheorem{definition}[theorem]{Definition}
\newtheorem{example}[theorem]{Example}

\theoremstyle{remark}
\newtheorem{remark}[theorem]{Remark}
\newtheorem{question}[theorem]{Question}

\numberwithin{equation}{section}
\newcommand{\Grs}{\Gr}
\newcommand{\GL}{\operatorname{GL}}
\newcommand{\Gr}{\operatorname{Gr}}
\newcommand{\Rep}{\operatorname{Rep}}
\newcommand{\Hom}{\operatorname{Hom}}
\newcommand{\PHom}{\operatorname{PHom}}
\newcommand{\Aut}{\operatorname{Aut}}
\newcommand{\End}{\operatorname{End}}
\newcommand{\Der}{\operatorname{Der}}
\newcommand{\Ext}{\operatorname{Ext}}
\newcommand{\ext}{\operatorname{ext}}
\newcommand{\E}{\operatorname{E}}
\newcommand{\e}{\operatorname{e}}
\newcommand{\Ho}{\operatorname{H}}
\newcommand{\Ker}{\operatorname{Ker}}
\newcommand{\Coker}{\operatorname{Coker}}
\newcommand{\Img}{\operatorname{Im}}
\newcommand{\T}{\operatorname{T}}
\newcommand{\N}{\operatorname{N}}
\newcommand{\Id}{\operatorname{Id}}
\newcommand{\rank}{\operatorname{rank}}
\newcommand{\Ind}{\operatorname{Ind}}
\newcommand{\Com}{\operatorname{Com}}
\newcommand{\proj}{\operatorname{proj-}\!}
\newcommand{\prj}{\operatorname{proj}}
\newcommand{\inj}{\operatorname{inj}}
\newcommand{\op}{\operatorname{op}}
\newcommand{\minus}{-}
\newcommand{\norm}[1]{\lVert#1\rVert}
\newcommand{\br}[1]{\overline{#1}}
\newcommand{\mc}[1]{\mathcal{#1}}
\newcommand{\Spec}{\operatorname{Spec}}
\renewcommand{\binom}[2]{{\textstyle {#1\choose #2}}}
\newcommand{\R}{{\mathbb R}}
\newcommand{\dashuparrow}{\rotatebox[origin=c]{90}{$\dashrightarrow$}}
\newcommand{\twoheaduparrow}{\rotatebox[origin=c]{90}{$\twoheadrightarrow$}}
\newcommand{\dv}{\underline{\dim}}

\begin{document}

\title{General Presentations of Algebras}

\author{Harm Derksen}
\address{Department of Mathematics, University of Michigan, Ann Arbor, MI 48109, USA}
\email{hderksen@umich.edu}
\thanks{The authors are partially supported by NSF grant DMS 901298.}

\author{Jiarui Fei}
\address{Department of Mathematics, University of Michigan, Ann Arbor, MI 48109, USA}
\email{jiarui@umich.edu}
\thanks{}


\date{}
\dedicatory{}

\begin{abstract}
For any finite dimensional basic associative algebra, we study the presentation spaces and their relation with the representation spaces. We prove two theorems about a general presentation, one on its subrepresentations and the other on its canonical decomposition.
As a special case, we consider rigid presentations. We show how to complete a rigid presentation and study the number of nonisomorphic direct summands and different complements. Based on that, we construct a simplicial complex governing the canonical decompositions of rigid presentations and provide some examples.
\end{abstract}

\maketitle

\section{Introduction}
Kac studied in \cite{K} properties of general representations of path algebras. We say that a {\em general representation} with dimension vector $\alpha$ has a certain property, if all representations in some Zariski open (and dense) subset of the space of $\alpha$-dimensional representations have that property.
We say that $\alpha=\alpha_1\oplus \alpha_2\oplus \cdots \oplus \alpha_r$ is the {\em canonical decomposition} of a dimension vector $\alpha$ if a general representation $V$ of dimension $\alpha$ has a decomposition $V=V_1\oplus V_2\oplus \cdots \oplus V_r$, where each $V_i$ is indecomposable of dimension $\alpha_i$. For such a decomposition, each $\alpha_i$ is a Schur root, which implies a general representation of dimension $\alpha_i$ is indecomposable.
Conversely, Kac showed that if $\alpha_1,\dots,\alpha_r$ are Schur roots, and $\ext(\alpha_i,\alpha_j)=0$ for all $i\neq j$, then $\alpha=\alpha_1\oplus\cdots\oplus \alpha_r$ is the canonical decomposition. Here, $\ext(\alpha,\beta)$ denotes $\dim \Ext^1(V,W)$ where $V$ and $W$ are general representations of dimension $\alpha$ and $\beta$ respectively.
In the same spirit, Schofield showed in \cite{S} that $\ext(\alpha,\beta)=0$ if and only if a general representation of dimension $\alpha+\beta$, has an $\alpha$ dimensional subrepresentation. Schofield also gave recursive formulas for $\ext(\alpha,\beta)$. An efficient algorithm for finding the canonical decomposition of a dimension vector was given in \cite{DW}. Moreover, in the same paper a simplicial complex governing the canonical decomposition is described.

A natural question is whether the results about general representations of a path algebra can be generalized to path algebras {\em with relations}. Unfortunately, the representation spaces of path algebras with relations may not be irreducible (or even reduced), so one has to work with general representations in a given irreducible component as in \cite{CBS}.
As another approach, we can study general {\em presentations}. For fixed projective representations $P_0,P_1$ of a given path algebra with relations, a {\em general presentation} $f:P_1\to P_0$ is a general element in the vector space $\Hom(P_1,P_0)$. This approach is justified by our result that general presentations present general representations (Theorem \ref{T:PP2Rep}). We restrict ourselves to the case when the algebra with relations is finite-dimensional. In that case the category of representations is Hom-finite and has ``enough projectives''. For two presentations $f,g$, we will define a finite-dimensional space $\E(f,g)$ which plays the role of $\Ext^1$ for path algebras (without relations).
We will prove in Theorems \ref{CDT} and \ref{SPT} analogs of the canonical decomposition and subpresentations of general presentations of path algebras {\em with} relations in this context. In the special case of path algebras (without relations), we recover the results of Kac and Schofield. Similar questions for general presentations of path algebras {\em without} relations were studied in \cite{IOTW} through a different approach.
Using presentation to study representations has a long history. The method was introduced by the Kiev school and formalized by Roiter to reduced the problem of studying representations of an algebra to that of a {\em bocs}. The technique was further developed by Crawley-Boevey in \cite{CB}.

A representation $M$ of a path algebra with relations is called {\em rigid} if $\Ext^1(M,M)=0$. A rigid representation $M$ is {\em partial tilting} if in addition it has projective dimension $\leq 1$.
Brenner-Butler's classical tilting theory \cite[VI]{ASS} says that any partial tilting representation $M$ can be completed to a maximal $\br{M}$ in the sense that any indecomposable representation $N$ for which $\br{M}\oplus N$ is partial tilting is isomorphic to a direct summand of $\br{M}$. Moreover, the number of pairwise nonisomorphic indecomposable summands of the completion $\br{M}$ is equal to $|Q_0|$, the number of vertices of the quiver of the path algebra. We will prove a similar statement in Theorem \ref{MRT} for {\em rigid presentations}, i.e., those presentation $f$ satisfying $\E(f,f)=0$. For a finite-dimensional path algebra, there are exactly two ways \cite[Proposition 3.6]{HU} to complete a {\em sincere almost tilting} representation $M$, that is, $M$ has $|Q_0|-1$ nonisomorphic indecomposable summands and each component of the dimension vector of $M$ is non-zero. We show in Proposition \ref{+-} that this is true for presentations of path algebras with relations as well. Moreover, we define a simplicial complex which governs the decomposition of rigid presentations. In the case of path algebra (without relations), this simplicial complex is the well-known {\em cluster complex} \cite{IOTW} associated to the cluster algebra of an acyclic quiver. {\em Cluster algebras} were introduced by Fomin and Zelevinsky in 2000 \cite{FZ}. Representation theory of path algebras with relations has been used to study the combinatorics of cluster algebras. In \cite{DWZ1,DWZ2}, Weyman, Zelevinsky and the first author use {\em quivers with potentials} to prove results about cluster algebras {\em with coefficients}. From the {\em potential}, one can derive a set of relations. In this context, the space $\E(f,g)$ is the same as the one defined in \cite{DWZ2}. If $V$ is a representation corresponding to a {\em cluster variable}, then $\E(f,f)=0$ where $f$ is the {\em minimal presentation} of $V$. It is unknown whether the converse is true.

This paper is organized as follows. After introducing the basic notation and background, we collect some interesting results concerning the projective presentations in Section \ref{S:PP}. Theorem \ref{T:PP2Rep} and its corollary are our first main results, which explain the relationship between space of presentations and space of representations. In Section \ref{S:SP} we introduce the space $\E$ and prove our second main result, Theorem \ref{SPT}, which is an analogue of Schofield's result on general subrepresentations. In Section  \ref{S:CD} we prove our third main result Theorem \ref{CDT}, which is an analogue of Kac's canonical decomposition. In Section \ref{S:RP}, we show how to complete a rigid presentation and prove another main result Theorem \ref{MRT} as an analogue of the classical tilting theory. Later we study the different complements to an almost maximal rigid presentation. In Section \ref{S:SC}, we introduce the simplicial complex $\mc{S}(A)$ and geometrically realize it on a sphere. Finally in Section \ref{sec:QP}, we briefly mention an application to quivers with potentials.

\section*{Notation and Background}
Throughout this paper, the base field $k$ is algebraically closed of characteristic 0.
Unless otherwise stated, unadorned $\Hom$ and $\otimes$ are all over $k$.
We use the shorthand $\hom_*(-,-)$ for the $k$-dimension for $\Hom_*(-,-)$, where $*$ is a given category.
Similar shorthands may be applied to other functors to $k$-vector spaces.

A {\em quiver} $Q$ is a pair $Q=(Q_0,Q_1)$ consisting of the set of vertices $Q_0$ and the set of arrows $Q_1$. If $a\in Q_1$ is an arrow, then $ta$ and $ha$ denote its tail and its head respectively. A path is a sequence of arrows $p = a_1a_2\cdots a_s$, with $ta_i=ha_{i+1}$ for all $i$. We define $tp=ta_s$ and $hp=ha_1$. For each vertex $v\in Q_0$ we also define the trivial path $e_v$ of length 0, satisfying $te_v=he_v=v$. An oriented cycle is a nontrivial path satisfying $hp=tp$. In this paper, we assume that $Q_0$ and $Q_1$ are finite but allow quivers to have oriented cycles. The path algebra $kQ$ is the $k$-vector space spanned by all paths. The multiplication in the algebra $kQ$ is defined as follows. If $p$ and $q$ are paths, then their product $p\cdot q$ is the concatenation of the paths if $tp=hq$, and is defined $0$ otherwise. The path algebra is bigraded: $kQ=\bigoplus_{v,w\in Q_0}e_vkQe_w$. A relation $r=\sum_{i=1}^s c_ip_i$ with $c_i\in k$ and $p_i$ a path, is called {\em admissible} if $r$ is homogeneous with respect to the grading, i.e., there exist $tr,hr\in Q_0$ such that $tp_i=tr$ and $hp_i=hr$ for all $i$. We assume that $I$ is an {\em admissible ideal}, i.e., a two sided ideal generated by admissible relations of length $\geqslant 2$. The path algebra {\em with relations} is the quotient algebra $kQ/I$. It contains a maximal semisimple subalgebra $R$ spanned by $\{e_v\}_{v\in Q_0}$. We assume that $A=kQ/I$ is finite dimensional henceforth.

A {\em dimension vector} $\alpha$ is a non-negative integer-valued function on $Q_0$. Given a dimension vector $\alpha$, let $M$ be the $R$-module $k^\alpha$, which is a family of finite-dimensional $k$-vector spaces $\{M(v)\}_{v\in Q_0}$ with $\dim M(v)=\alpha(v)$. A {\em representation} $M$ of $kQ$ is an $R$-module $M$ together with a family of linear maps $\{M(a):M(ta)\to M(ha)\}_{a\in Q_1}$. Fixed a dimension vector $\alpha$, the representation space $\Rep_\alpha(kQ)$ of all $\alpha$-dimensional representations is the vector space $\bigoplus_{a\in Q_1}\Hom(k^{\alpha(ta)},k^{\alpha(ha)}).$
For any path $p = a_1a_2\cdots a_s$, we define $M(p): M(tp)\to M(hp)$ to be the composition $M(a_1)M(a_2)\cdots M(a_s)$. So the assignment $M\mapsto M(p)$ defines a polynomial map
$F_p:\Rep_{\alpha}(kQ)\to \Hom(k^{\alpha(tp)},k^{\alpha(hp)}).$
A representation $M$ of $A=kQ/I$ is a representation of $kQ$ satisfying all the relations in $I$, i.e., $M(r)=0$ for all $r\in I$.
Let $k[\Rep_{\alpha}(kQ)]$ denote the ring of polynomial functions on $\Rep_{\alpha}(kQ)$. Then $F_p$ is represented by an $\alpha(hp)\times\alpha(tp)$ matrix with entries in $k[\Rep_{\alpha}(kQ)]$. Let $\tilde{I}\subseteq k[\Rep_{\alpha}(kQ)]$ be the ideal generated by the entries of all $F_r$ for which  $r\in I$. The representation space $\Rep_{\alpha}(A)$ is the scheme $\Spec(k[\Rep_{\alpha}(kQ)]/\tilde{I})$. The coordinate ring $k[\Rep_{\alpha}(kQ)]/\tilde{I}$ represents the following functor: $B\mapsto\Hom_{R\text{-alg}}(A,\End_B(B\otimes M))$, from the category of finitely generated commutative $k$-algebra to the Sets. As a set, $\Rep_{\alpha}(A)$ consists of all $\alpha$-dimensional representation of $A$.

The group $\GL_\alpha:=\prod_{v\in Q_0}\GL_{\alpha(v)}$ acts on $\Rep_\alpha(A)$ by the natural base change. Two representations $M,N\in\Rep_\alpha(A)$ are isomorphic if they lie in the same $\GL_\alpha$-orbit.
A morphism $f:M\to N$ between two representations is a collection of linear maps $\{f(v):M(v)\to N(v)\}_{v\in Q_0}$ such that for each $a\in Q_1$ we have $N(a)f(ta)=f(ha)M(a)$. The category $\Rep(A)$ of representations of $A$ is equivalent to the category of all finite dimensional left $A$-modules. For each vertex $v\in Q_0$, let $S_v$ be the 1-dimensional simple representation with $S_v(v)=k$. The Grothendieck group $K_0(\Rep(A))$ of this abelian category is a free abelian group generated by all $S_v$ \cite[Theorem III.3.5]{ASS}. So the dimension vector of a representation $M$ is its class in $K_0(\Rep(A))$. A representation $M\in\Rep(A)$ is called {\em indecomposable} if $M=L\oplus N$ for $L,N\in\Rep(A)$, then one of $L$ and $N$ must be the zero objects. The category $\Rep(A)$ has the {\em Krull-Schmidt property}, meaning that each representation has a unique decomposition into indecomposable summands.

The category $\Rep(A)$ has enough projective objects.
All indecomposable projective representations are of form $P_v:=Ae_v$ for some vertex $v\in Q_0$.
They are characterized by the property that $\Hom_A(P_v,N)=N(v)$ for any $N\in\Rep(A)$.
Unfortunately when $v=0,1$ this standard notation may be confused with $P_0,P_1$ in some presentation $P_0\xrightarrow{f} P_1$,
but we feel that this should not be an issue in this paper.
For any $\beta\in\mathbb{N}_0^{Q_0}$, we denote $P(\beta):=\bigoplus_{v\in Q_0} P_v^{\beta(v)}$ and define the space of presentation  $\PHom_A(\beta_1,\beta_0):=\Hom_A(P(\beta_1),P(\beta_0)).$ The automorphism group $\Aut_A(M)$ of $M\in\Rep(A)$ consists of invertible elements in $\Hom_A(M,M)$. The group $\Aut_A(P(\beta_1))\times\Aut_A(P(\beta_0))$ acts on $\PHom_A(\beta_1,\beta_0)$ by $(g_1,g_0)f=g_0fg_1^{\minus 1}$.
\begin{definition} \label{D:delta}
We define the $\delta$-$vector$ of a presentation $f\in\PHom_A(\beta_1,\beta_0)$ by $\delta(f)=\beta_0-\beta_1\in\mathbb{Z}^{Q_0}$.
Conversely, to any $\delta\in\mathbb{Z}^{Q_0}$, we associate the {\em reduced} presentation space $\PHom_A(\delta)=\PHom_A([-\delta]_+,[\delta]_+)$,
where $[\delta]_+(v)=\max\{\delta(v),0\}$.
\end{definition}

The homotopy category $K^b(\proj A)$ of bounded complexes of projective representations of $A$ is a triangulated category with the Grothendieck group $K_0(K^b(\proj A))$ isomorphic to $K_0(\Rep(A))$. In fact, if we embed $K^b(\proj A)$ and $\Rep(A)$ canonically in the bounded derived category $D^b(\Rep(A))$, then the Euler form
$\langle P,M\rangle=\chi(\text{RHom}_{D^b(\Rep(A))}(P,M))\text{ on }K^b(\proj A)\times\Rep(A)$
gives us a dual pairing with the class of $P_v$ dual to the class of $S_v$. The $\delta$-vector of a presentation is nothing but its class in $K_0(K^b(\proj A))$. Let $f=P_l\to\cdots\to P_1\to P_0$ be a length $l+1$ complex in $K^b(\proj A)$. The $i$-{\em th Betti vector} $\beta_i(f)$ is the minimal vector such that $f$ is equivalent to some complex $P(\beta_l(f))\to\cdots\to P(\beta_1(f))\to P(\beta_0(f))$ in $K^b(\proj A)$.

\section{Presentations vs. Representations}  \label{S:PP}
To begin with, we need a quick review on derivations. Let $A=kQ/I$ be a finite dimensional algebra and $R$ be the maximal semisimple subalgebra spanned by all $e_v\in kQ$. Recall that if $B$ is an $A$-$A$ bimodule, an $R$-derivation $d:A\to B$ is a linear map such that $d(aa')=ad(a')+d(a)a'$ for all $a,a'\in A$ and $d(R)=0$. A derivation $d$ is called {\em inner} if there is some $b\in B$ such that $d(a)=ab-ba$ for all $a\in A$. We denote by $\Der_R(A,B)$ the space of all $R$-derivations $A\to B$ and $\Der_R^0(A, B)$ the subspace of inner $R$-derivations. For any (left) $A$-modules $L$ and $N$, $\Hom(N,L)$ has a natural $A$-$A$ bimodule structure. If $d\in\Der_R(A,\Hom(N,L))$, then let $a'=e_{ta}$, we get $d(a)=d(ae_{ta})=d(a)e_{ta}$ and similarly $d(a)=d(e_{ha}a)=e_{ha}d(a)$. So each $d(a)$ can be identified with an element in $\Hom(N(ta),L(ha))$ and if $d$ is inner, then $d(a)=F(ha)N(a)-L(a)F(ta)$ for some $F\in\Hom_R(N,L)$.
The following fact is proved in \cite[Section 3.7]{LB} for the absolute case, i.e., over $k$. One can slightly modify that proof to obtain

\begin{lemma}\label{HF} The tangent space $\T_M\Rep_\alpha(A)$ of the scheme $\Rep_{\alpha}(A)$ at $M$ is isomorphic to $\Der_R(A,\End(M))$, and the tangent space $\T_M(\pmb{O}_M)$ of the orbit $\pmb{O}_M=\GL_\alpha\cdot M$ is isomorphic to $\Der_R^0(A,\End(M))$. Moreover,
$$\Ext_A^1(N,L)=\Der_R(A,\Hom(N,L))/\Der_R^0(A,\Hom(N,L)).$$
In particular, the normal space $\N_M$ to $\pmb{O}_M$ inside $\Rep_\alpha(A)$ is isomorphic to
$\Ext_A^1(M,M)$.
\end{lemma}

Any projective presentation $f\in\Hom_A(P_1,P_0)$ induces an exact sequence
\begin{equation}\label{ES2}0\to K\xrightarrow{\iota}P_1\xrightarrow{f}P_0\xrightarrow{\pi}C\to0.\end{equation}
In this situation, we say that $C$ is presented by $f$.
Apply the bifunctor $\Hom_A(-,-)$ to $\eqref{ES2}$ and itself, then we obtain the following double complex:
\begin{displaymath} \begin{matrix}
\Hom_A(C,C)\hookrightarrow & \Hom_A(P_0,C)\xrightarrow{f^C} & \Hom_A(P_1,C)\xrightarrow{\iota^C} & \Hom_A(K,C)\\
\uparrow & \twoheaduparrow_{} & \quad\twoheaduparrow_{\pi_*} & \uparrow\\
\Hom_A(C,P_0)\hookrightarrow & \Hom_A(P_0,P_0)\xrightarrow{f^*} & \Hom_A(P_1,P_0)\xrightarrow{\quad} & \Hom_A(K,P_0)\\
\uparrow & \uparrow & \quad\uparrow_{f_*} & \uparrow_{}\\
\Hom_A(C,P_1)\hookrightarrow & \Hom_A(P_0,P_1)\xrightarrow{\quad} & \Hom_A(P_1,P_1)\xrightarrow{\quad} & \Hom_A(K,P_1)\\
\end{matrix} \end{displaymath}

Let $\Hom_A(P_0,P_1)\xrightarrow{d_0}\Hom_A(P_1,P_1)\oplus\Hom_A(P_0,P_0)\xrightarrow{d_1}\Hom_A(P_1,P_0)$ be the induced complex.
Recall that the group $\Aut_A(P_1)\times\Aut_A(P_0)$ acts on $\PHom_A(P_1,P_0)$ by $(g_1,g_0)f=g_0fg_1^{\minus 1}$.
If we identify the Lie algebra of $\Aut_A(P_1)\times\Aut_A(P_0)$ with $\End_A(P_1)\oplus\End_A(P_0)$, then it is easy to see that $d_1=(-f_*,f^*)$ is the differential of the orbit map at $f$. So the image of $d_1$ gives the tangent space of the $\Aut_A(P_1)\times\Aut_A(P_0)$-orbit $\mathcal{O}_f$ of $f$, and thus the normal space $\N_f$ of $\mathcal{O}_f$ in $\Hom_A(P_1,P_0)$ is the quotient space $\Hom_A(P_1,P_0)/\Img d_1$. An easy diagram chasing can show that it can be identified with $\Hom_A(P_1,C)/\Img f^C$ under $\pi_*$. In the meanwhile, the normal space $\N_C$ of $\GL_\alpha\cdot C$ in $\Rep_\alpha(A)$ is $\Ext_A^1(C,C)\cong\Ker\iota^C/\Img f^C$.
So the normal space $\N_C$ can be identified as a subspace of the normal space $\N_f$.

For any $\Aut_A(P_1)\times\Aut_A(P_0)$-stable subvariety $Y$ of $\Hom_A(P_1,P_0)$ and $\GL_\alpha$-stable subvariety $X$ of $\Rep_\alpha(A)$, we define
\begin{align*}Z(Y,X)=\{(f,\pi,C)\in Y\times\Hom_R(P_0,k^\alpha)\times X \mid & \pi\in\Hom_A(P_0,C)\\
             \text{ and } & P_1\xrightarrow{f}P_0\xrightarrow{\pi}C\to 0 \text{ is exact }\}.\end{align*}
\hbox{Let $p$ (resp. $q$) be the projection of $Z(Y,X)$ given by $z=(f,\pi,C)\mapsto C$ (resp. $\mapsto f$).}

\begin{lemma} The following sequence is exact.
\begin{equation*}\T_z Z(Y,X)\xrightarrow{dp}\T_C X\to\genfrac{}{}{}{}{\N_C X}{\N_f Y\cap\N_C X}\to 0.\end{equation*}
\end{lemma}

\begin{proof}
Using the algebra of dual numbers $k[\epsilon]$, we compute the tangent space $\T_z Z(Y,X)$ as follows.
The condition $\pi\in\Hom_A(P_0,C)$ is defined by $C(a)\pi(ta)=\pi(ha)P_0(a)$ for all $a\in Q_1$.
Since the rank function is upper-semicontinuous, the exactness of the defining sequence of $Z(Y,X)$ is equivalent to $\pi f=0$ plus an open condition.
So a triple $(F,G,\gamma)\in\T_f Y\times\Hom_R(P_0,C)\times\T_C X$ lies in $\T_z Z(Y,X)$
\begin{align} \label{eq1} & \iff \begin{cases}
(\pi+\epsilon G)(f+\epsilon F)=0,\\
(C(a)+\epsilon\gamma(a))(\pi(ta)+\epsilon G(ta))=(\pi(ha)+\epsilon G(ha))P_0(a).\end{cases}\notag \\
& \iff \begin{cases}
Gf+\pi F=0,\\
\gamma(a)\pi(ta)=G(ha)P_0(a)-C(a)G(ta).\end{cases}\end{align}

For any $\gamma\in\T_C X\subseteq\Der_R(A,\End(C))$, let $\xi_\gamma$ be the corresponding element in $\Ext_A^1(C,C)$, then $\xi_\gamma$ has a representative $h\in\Ker\iota_C$. A simple diagram chasing can show that the exactness at $\T_C X$ is equivalent to that $\gamma\in\Img dp$ if and only if $h$ can be chosen in $\pi_*\T_fY$.

Suppose that $\gamma\in \Img dp$, then it satisfies \eqref{eq1}. We have that $Gf=-\pi F\in\pi_*\T_fY\subseteq\Hom_A(P_1,C)$ and it lies in $\Ker\iota_C$. We claim that $h=Gf$ represents $\xi_\gamma$. Consider the following commutative diagram, where $\left(\begin{smallmatrix}C \\ \ \uparrow_\gamma \\ C \end{smallmatrix}\right)$ is the representation obtained by extension using $\gamma$ as in Lemma \ref{HF}.
\begin{displaymath} \begin{matrix}
& & P_1 & \xrightarrow{\ f\ } & P_0 & \xrightarrow{\ \pi\ } &  C & \to & 0\\
& & \quad\Big\downarrow{^{Gf}} & & \quad\Big\downarrow{\binom{G}{\pi}} & & \Big\Arrowvert \\
0 & \xrightarrow{} & C & \xrightarrow{\binom{0}{\Id}} & \quad\left(\begin{smallmatrix}C \\ \ \uparrow_\gamma \\ C \end{smallmatrix}\right)\quad &\xrightarrow{(0,\Id)} & C & \to & 0
\end{matrix} \end{displaymath}
Note that the second equation of \eqref{eq1} says that the map $\binom{G}{\pi}$ in the middle is a morphism of representations. Now the claim follows easily by applying $\Hom_A(-,C)$ to the upper row.


Conversely, suppose that $h=\pi F$ representing $\xi_\gamma$ for some $F\in\T_f Y$. Since $\Ext_A^1(P_0,C)=0, \{\gamma(a)\pi(ta)\}_{a\in Q_1}\in\Der_R(A,\Hom(P_0,C))$ is inner, that is,
$$\gamma(a)\pi(ta)=G'(ha)P_0(a)-C(a)G'(ta)\text{ for some }G'\in\Hom_R(P_0,C).$$ Then
\begin{align*}G'f(ha)P_1(a)-C(a)G'f(ta) & =G'(ha)P_0(a)f(ta)-C(a)G'(ta)f(ta)\\
& =\gamma(a)\pi(ta)f(ta)=0,\end{align*} so $G'f\in\Ker\iota_C\subseteq\Hom_A(P_1,C)$. By the same diagram replacing $Gf$ by $G'f$, we see that $G'f$ also represents $\xi_\gamma\in\Ext_A^1(C,C)$. Hence $G'f-h\in\Img f^C$. So we can write $h$ as $h=Gf$ with $G'-G\in\Hom_A(P_0,C)$. Finally we get the second equation of \eqref{eq1} $$\gamma(a)\pi(ta)=G'(ha)P_0(a)-C(a)G'(ta)=G(ha)P_0(a)-C(a)G(ta).$$ Therefore, $\gamma\in\Img dp$.
\end{proof}

Consider
$$\xymatrix{
& Z(\Hom_A(P_1,P_0),\Rep_\alpha(A)) \ar[ld]_{q}\ar[rd]^{p} &\\
\Hom_A(P_1,P_0)\quad& & \Rep_\alpha(A)
}$$

\begin{theorem} \label{T:PP2Rep}
The projection $p$ is open.
\end{theorem}

\begin{proof}
Let $Y=\Hom_A(P_1,P_0)$, then for any $\GL_\alpha$-stable subvariety $X$ of $\Rep_\alpha(A)$, $\T_z Z(Y,X)\xrightarrow{dp}\T_C X$ is surjective because $\N_fY=\Hom_A(P_1,C)/\Img f^C\supseteq\N_C X$. Hence, the projection $p:Z(Y,X)\to X$ is dominant.

Now let $X=\Rep_\alpha(A)$, then by a theorem of Chevalley $\Img p$ is constructible. We need to prove $\Img p$ is open in $X$. This is now equivalent to showing that $\Img p$ is stable under any generization, namely, if $x$ is a point in the scheme $X$ and $\Img p\cap \overline{\{x\}}$ is non-empty, then $x\in\Img p$ \cite[Exercise III.3.18]{Ha}. In fact, it is not hard to show that it is sufficient to consider the $\GL_\alpha$-stable points of $X$. By the above lemma, we know that $\Img p\cap \overline{\{x\}}$ is dense in $\overline{\{x\}}$, so $\Img p$ must contain the generic point $x$ of $\overline{\{x\}}$.
\end{proof}

\begin{lemma}\
The projection $q$ is a principle $\GL_\alpha$-bundle over its image.
\end{lemma}

\begin{proof} The same proof as \cite[Corollary 3.4]{R} also works here.
\end{proof}

\begin{lemma}
Given any two presentations $f,f'\in\Hom_A(P_1,P_0)$ of $C$ and an automorphism $g$ of $C$, we can complete the diagram with $g_i\in\Aut_A(P_i)$~for~$i=0,1$.
\begin{displaymath} \begin{matrix}
P_1 & \xrightarrow{f} & P_0 & \xrightarrow{\pi} & C & \to & 0\\
\quad\dashuparrow{_{g_1}} & & \quad\dashuparrow{_{g_0}} & & \ \ \big\uparrow{_g} &\\
P_1 & \xrightarrow{f'} & P_0 & \xrightarrow{\pi'} & C & \to & 0\\
\end{matrix} \end{displaymath}
\end{lemma}

\begin{proof} $P_0\xrightarrow{\pi} C\to 0$ can be split into $P_0'\oplus P_0''\xrightarrow{(\pi',0)} C\to 0$, where $P_0'$ is the projective cover of $C$. Then the map $g\pi'$ induces a morphism $P_0\xrightarrow{g_0'} P_0'$, which is surjective because $P_0'$ is the projective cover. The kernel of $g_0'$ must be isomorphic to $P_0''$, so it induces a projection $g_0'':P_0\to P_0''$. We set $g_0$ to be $\binom{g_0'}{g_0''}$, which is an isomorphism. Apply the same construction to $P_1 \xrightarrow{f} \Ker(\pi)$, then we get an isomorphism $g_1$.
\end{proof}

\begin{corollary} \label{PPRep}
The map $pq^{-1}$ gives a bijection between $\Aut_A(P_1)\times\Aut_A(P_0)$-stable subvariety of $\Img q$ and $\GL_\alpha$-stable subvariety of $\Img p$, preserving openness, closure, and irreducibility.
\end{corollary}

We need the rest of this section sporadically, in fact mostly in the last section, so readers can come back here whenever needed. Among all presentations of $M\in\Rep_\alpha(A)$, the minimal projective presentation $P(\beta_1(M))\to P(\beta_0(M))\to M\to 0$ may be the most interesting one. Let us give a concrete way to compute the Betti-vectors $\beta_0(M)$ and $\beta_1(M)$ for any $M\in\Rep(A)$. Suppose that $\cdots\to P_2\to P_1\to P_0\to M\to 0$ is the minimal resolution of $M$. Apply $\Hom_A(-,S_v)$ to the resolution and we get \begin{equation*}0\to\Hom_A(M,S_v)\to\Hom_A(P_0,S_v)\xrightarrow{d_1}\Hom_A(P_1,S_v)\xrightarrow{d_2}\Hom_A(P_2,S_v)\to\cdots.\end{equation*}
The minimality of the resolution implies that $d_i$ are all zero maps. So $$\beta_i(M)_v=\hom_A(P_i,S_v)=\ext_A^i(M,S_v).$$

To proceed further, we need the canonical injective resolution of $S_v$. Suppose that $A=kQ/I$ and $Q_2$ be a basis of $\genfrac{}{}{}{}{I}{JI+IJ}$, where $J$ is the Jacobson radical~of~$A$.
Then $A$ has the following canonical (or simplicial)~$A$-$A$-bimodule~resolution~\cite[(1.4)]{BK}:
$$\cdots\to\bigoplus_{r\in Q_2}Ae_{hr}\otimes e_{tr}A\xrightarrow{d_1}\bigoplus_{a\in Q_1}Ae_{ha}\otimes e_{ta}A\xrightarrow{d_0}\bigoplus_{v\in Q_0}Ae_v\otimes e_v A\xrightarrow{\mu}A\to 0,$$ where $\mu$ is the multiplication of $A$, $d_0$ is defined by $d_0(e_{ha}\otimes e_{ta})=a\otimes e_{ta}-e_{ha}\otimes a,$ and $d_1$ is defined by $d_1(e_{hr}\otimes e_{tr})=\sum_{a\in Q_1}{r_2^a\otimes r_1^a}$, where $r=r_2^a a r_1^a$.
Tensoring the above resolution with $M$ yields the canonical projective resolution of $M$:
\begin{equation}\label{canproj} \cdots\to\bigoplus_{r\in Q_2}P_{hr}\otimes M(tr)\xrightarrow{}\bigoplus_{a\in Q_1}P_{ha}\otimes M(ta)\xrightarrow{}\bigoplus_{v\in Q_0}P_v\otimes M(v)\xrightarrow{}M\to 0.\end{equation}
Applying the dual construction to $M=S_v$, we obtain the canonical injective resolution of $S_v$:
\begin{equation*}0\to S_v\to I_v\xrightarrow{\varphi^*}\bigoplus_{ha=v}I_{ta}\xrightarrow{\psi^*}\bigoplus_{hr=v}I_{tr}\to\cdots.\end{equation*}
Here, $I_v$ is the indecomposable injective representation corresponding to the vertex $v$, $\varphi^*$ is given by $\varphi^*(p^*)=p^*(a)$,
and the restriction of $\psi^*$ on $I_{ta}$ is $\psi^*(p^*)=p^*(a^{\minus 1}r)$, where $a^{\minus 1}$ is the formal inverse of $a$. Apply $\Hom_A(M,-)$ to this resolution and take the trivial dual, we obtain the complex
\begin{equation}\label{alphagamma}\cdots\to\bigoplus_{hr=v}M(tr)\xrightarrow{\psi}\bigoplus_{ha=v}M(ta)\xrightarrow{\varphi}M(v).\end{equation}
We conclude that $\beta_0(M)_v=\dim\Coker\varphi$ and $\beta_1(M)_v=\dim\Ker\varphi/\Img\psi$.

Let $K^2(\proj A)$ be the set of all presentations in $K^b(\proj A)$, then all the indecomposable objects in $K^2(\proj A)$ are in one-to-one correspondence with $P_v[1]$ and minimal projective presentations of $M\in\Rep(A)$. We define the {\em AR-transformation} $\tau$ on $K^2(\proj A)$ as follows. For any indecomposable $f\neq P_v,P_v[1]$, we define $\tau(f)$ to be the minimal projective presentation of $\tau(\Coker(f))$, the classical Auslander-Reiten transformation \cite[Definition IV.2.3]{ASS} of $\Coker(f)$; define $\tau(P_v)=P_v[1]$ and $\tau(P_v[1])$ to be the minimal projective presentation of $I_v$. Then we can extend the definition additively to the whole $K^2(\proj A)$. It follows immediately from the classical theory \cite[Proposition IV.2.10]{ASS} that $\tau$ is bijective on $K^2(\proj A)$ and let $\tau^{\minus 1}$ be its inverse. In this way, the preprojective and preinjective components of this algebra are connected pictorially by $$\xymatrix{ \cdots \ar@{--}[r] & P_v \ar@{--}[r] & P_v[1] \ar@{--}[r] & ({\rm projective\ presentation\ of\ }I_v) \ar@{--}[r] & \cdots. }$$

\section{Subpresentations of a general presentation} \label{S:SP}

We first briefly recall the construction of the quiver Grassmannian. For dimension vectors $\alpha=\beta+\gamma$, we define $\Gr\binom{\alpha}{\gamma}:=\prod_{v\in Q_0}\Gr\binom{\alpha(v)}{\gamma(v)}$, where $\Gr\binom{\alpha(v)}{\gamma(v)}$ is the usual Grassmannian variety of $\gamma(v)$-dimensional subspaces of $k^{\alpha(v)}$.
We define the variety
\begin{gather*}Z\binom{\alpha}{\gamma}=\{(W,V')\in \Rep_{\alpha}(A) \times\Gr\binom{\alpha}{\gamma}\mid V'\text{ is a subrepresentation of }W\},\\
\intertext{and for fixed $V\in\Rep_\gamma(A)$ and $U\in\Rep_\beta(A)$, its subvariety}
Z_V^U=\{(W,V')\in Z\binom{\alpha}{\gamma}\mid V'\cong V,W/V'\cong U\}.\end{gather*}

Let $p:Z\binom{\alpha}{\gamma}\to\Rep_{\alpha}(A)$ be the first factor projection $(W,V')\mapsto W$.
\begin{definition}
We define the {\em quiver Grassmannian} $\Gr_\gamma(W)$ to be the scheme-theoretic fibre $p^{\minus 1}(W)$, and its subvarieties $\Gr_V^U(W)=(p|_{Z_V^U})^{\minus 1}(W)$.
\end{definition}

\begin{lemma} \label{L:QG}
Let $P,P^s,P^q$ be projective representations of $A$, then $\Gr_{P^s}^{P^q}(P)$ is an open and irreducible subvariety of the smooth variety
$\Gr_{\gamma}(P)$, where $\gamma$ is the dimension vector of $P^s$.
\end{lemma}

\begin{proof} It is known \cite[Lemma 3.2]{S} that the tangent space of $\Gr_\gamma(W)$ at $V$ can be identified with $\Hom_A(V,W/V)$. In our situation, $\hom_A(P^s,P/P^s)$ is constant, so $\Gr_{\gamma}(P)$ is smooth.

There is an open dense subset $\Hom_A(P,P^q)^\gamma$ of $\Hom_A(P,P^q)$, in which the morphisms have the {\em general corank} $\gamma$.
So the image of the map $\Hom_A(P,P^q)^\gamma\to\Gr_\gamma(P)$ sending a presentation to its kernel is irreducible. But the image is exactly $\Gr_{P^s}^{P^q}(P)$, since the kernel of any epimorphism to a projective representation must be projective.

It is proved in \cite[Corollary II.5]{F} that the tangent space of $\Gr_{P^s}^{P^q}(P)$ at $P^s$ can be identified with the image of the surjective map $\Hom_A(P,P^q)\xrightarrow{\Hom_A(\iota,P^q)}\Hom_A(P^s,P^q)$, where $\iota$ is any embedding of $P^s$ into $P$. Hence, $\Gr_{P^s}^{P^q}(P)$ is open in $\Gr_{\gamma}(P)$.
\end{proof}

Given any projective presentation $P_1\xrightarrow{f} P_0$, a presentation $P_1'\xrightarrow{f'} P_0'$ is called {\em subpresentation} of $f$ if it is a subcomplex of $f$:
$$\xymatrix{
P_1\ar[r]^f & P_0\\
P_1'\ar@{^(->}[u]\ar[r]_{f'} & P_0'\ar@{^(->}[u]}$$
\noindent We call $f'$ a {\em genuine} subpresentation if $P_1''=P_1/P_1'$ and $P_0''=P_0/P_0'$ are both projective.

Given any two projective presentations $f'\in\Hom_A(P_1',P_0')$ and $f''\in\Hom_A(P_1'',P_0'')$ with cokernel $L$ and $N$ respectively,
there is an induced double complex with exact rows and columns:

\begin{equation}\label{eq:double}
\begin{matrix}
\Hom_A(L,N)\hookrightarrow & \Hom_A(P_0',N)\xrightarrow{f_N'} & \Hom_A(P_1',N)\\
\uparrow_{} & \twoheaduparrow & \twoheaduparrow\\
\Hom_A(L,P_0'')\hookrightarrow & \Hom_A(P_0',P_0'')\xrightarrow{f_0'} & \Hom_A(P_1',P_0'')\\
\uparrow & \quad\uparrow_{f_0''} & \ \quad\uparrow_{\minus f_1''}\\
\Hom_A(L,P_1'')\hookrightarrow & \Hom_A(P_0',P_1'')\xrightarrow{f_1'} & \Hom_A(P_1',P_1'')
\end{matrix} \end{equation}
Here $f_N'=\Hom_A(f',N),f_i'=\Hom_A(f', P_i'')$, and $f_i''=\Hom_A(P_i',f'').$


\begin{definition}
We define $\E(f',f''):=\Hom_{K^b(\proj A)}(f'[\minus 1],f'')=\Coker (f_0',\minus f_1'')$. In addition, we denote $\Ker (f_0',\minus f_1'')$ by $\Ho(f',f'')$.
\end{definition}

\noindent An easy diagram chasing can show:

\begin{lemma} \label{DC}
$\E(f',f'')\cong\Coker f_N'$. In particular, $\E(f',f'')\supseteq\Ext_A^1(L,N)$ and if $f'$ is injective, then they are equal.
\end{lemma}
\noindent So by definition $\E(f',f'')$ is homotopy invariant, and by the above lemma it depends on $N=\Coker f''$ rather than $f''$ itself.


For projective representations $P_1,P_0,P_1^s,P_0^s,P_1^q,P_0^q$, we denote $\Grs_\oplus:=\Gr_{P_1^s}^{P_1^q}(P_1)\times \Gr_{P_0^s}^{P_0^q}(P_0)$. To make our notation more compact, we will always write $P_i''$ for $P_i/P_i'$.
We define
$$Z=\{(f,P'_1,P'_0)\in\Hom_A(P_1,P_0)\times \Grs_\oplus\mid f(P'_1)\subseteq P'_0\}.$$
Let $p_1:Z \to \Hom_A(P_1,P_0)$  be the first factor projection and $p_2:Z \to \Grs_\oplus$ be the second one:
$$
\xymatrix{
& Z\ar[ld]_{p_1}\ar[rd]^{p_2} &\\
\Hom_A(P_1,P_0)\quad& & \Grs_{\oplus}}
$$
Any point $x=(f,P_1',P_0')\in Z$ determines the following diagram up to automorphisms of $P_1',P_0',P_1'',P_0''$.
\begin{displaymath} \begin{matrix}
0 & \to & P_0' & \xrightarrow{\iota_0} & P_0 & \xrightarrow{\pi_0} & P_0'' & \to & 0\\
& & \quad\uparrow_{f'} & & \ \ \uparrow_{f} & & \quad\uparrow_{f''}\\
0 & \to & P_1' & \xrightarrow{\iota_1} & P_1 & \xrightarrow{\pi_1} & P_1'' & \to & 0\\
\end{matrix} \end{displaymath}

\begin{lemma} \label{bundle}
The space $Z$ is a vector bundle over $\Grs_\oplus$ with fibre $p_2^{\minus 1}(P'_1,P'_0)$ isomorphic to $\Hom_A(P_1',P_0')\oplus\Hom_A(P_1'',P_0)$.
Hence $Z$ is smooth and irreducible of dimension equal to
\begin{equation*}(\hom_A(P_1^s,P_1^q)+\hom_A(P_0^s,P_0^q))+(\hom_A(P_1,P_0)-\hom_A(P_1^s,P_0^q)).\end{equation*}
\end{lemma}

\begin{proof}
$Z$ can be realized as the subbundle of the trivial vector bundle
$\Hom_A(P_1,P_0)$ on $\Grs_{\oplus}$
$$Z=\{(P_1',P_0',f)\in \Grs_{\oplus} \times \Hom_A(P_1,P_0)\mid \pi_0f\iota_1=0\}.$$
Then the fibre $p_2^{\minus 1}(P'_1,P'_0)$ is given by the kernel of $$\Hom_A(P_1,P_0)\to \Hom_A(P_1,P_0'')\to \Hom_A(P_1',P_0''),$$ which is isomorphic to $\Hom_A(P_1',P_0')\oplus\Hom_A(P_1'',P_0)$.
Then the dimension of  $Z$ is equal to $\dim \Grs_\oplus+\dim p_2^{\minus 1}(P'_1,P'_0)$.
We recall from Lemma \ref{L:QG} that $\dim \Gr_{P^s}^{P^q}(P) = \hom_A(P^s,P^q)$, so this is what we desired.
\end{proof}

\begin{lemma} \label{tangent}
The tangent space $\T_xp_1^{\minus 1}(f)\text{ is isomorphic to }\Ker(f_0',\minus f_1'')$.
\end{lemma}

\begin{proof}
For any $(F,\varphi_1,\varphi_0)\in\T_x(\Hom_A(P_1,P_0)\times \Grs_\oplus)$\\
\rightline{$\cong\Hom_A(P_1,P_0)\oplus\Hom_A(P_1',P_1'')\oplus\Hom_A(P_0',P_0''),\qquad $}\\
let $\tilde{\varphi}_i\in\Hom_A(P'_i,P_i)$ be some lifting of $\varphi_i$, and $\tilde{\varphi}_i$ can be further lifted to some $\bar{\varphi}_i\in\End_R(P_i)$. Using the algebra of dual numbers $k[\epsilon]$, we compute the tangent space $\T_xZ$ as follows. $(F,\varphi_1,\varphi_0)\in\T_xZ$
\begin{align*}
&\iff (f+\epsilon F)(\Id+\epsilon\bar{\varphi}_1)P_1'\subseteq(\Id+\epsilon\bar{\varphi}_0)P_0'\\
&\iff(\Id-\epsilon\bar{\varphi}_0)(f+\epsilon F)(\Id+\epsilon\bar{\varphi}_1)P'_1\subseteq P'_0\\
&\iff(f\bar{\varphi}_1-\bar{\varphi}_0 f+F)P'_1\subseteq P'_0\\
&\iff(f\tilde{\varphi}_1-\tilde{\varphi}_0 f'+F)P'_1\subseteq P'_0\\
&\iff\pi_0(f\tilde{\varphi}_1-\tilde{\varphi}_0 f'+F)P'_1=0\\
&\iff(f''\varphi_1-\varphi_0 f'+\pi_0F)P'_1=0\\
&\iff\pi_0F\iota_1=\varphi_0 f'-f''\varphi_1
\end{align*}
Hence $\T_xp_1^{\minus 1}(f)\cong\Ker(f_0',-f_1'')$.
\end{proof}

\begin{corollary}
The function $Z\to\mathbb{Z^+}$ given by $(f,P_1',P_0')\mapsto\dim\E(f',f'')$ is upper semi-continuous.
\end{corollary}

In the mean while, the function $\e(-,-):=\dim\E(-,-)$ on $\Hom_A(P_1',P_0')\times\Hom_A(P_1'',P_0'')$ is also upper semi-continuous because it is the corank function of a morphism of vector bundles on $\Hom_A(P_1',P_0')\times\Hom_A(P_1'',P_0'')$. Since for any $(f',f'')$ there is a point $x\in Z$ corresponding to it, the minimal values of these two functions must coincide.

It is proved in \cite[Theorem V.2.2]{IOTW} and also follows from Corollary \ref{CDC} later that for any $\beta_1,\beta_0\in\mathbb{N}^{Q_0}$ a general presentation in $\PHom_A(\beta_1,\beta_0)$ is homotopy equivalent to a general presentation in the reduced space $\PHom_A(\beta_0-\beta_1)$ (Definition \ref{D:delta}), so the following definition makes sense. Let $\e(\beta_0'-\beta_1',\beta_0''-\beta_1'')$ be the minimal value of $\e(-,-)$ on $\PHom_A(\beta_1',\beta_0')\times\PHom_A(\beta_1'',\beta_0'')$.

\begin{proposition}\label{SPP} For any $\beta_i',\beta_i''\in\mathbb{N}_0^{Q_0}$ $(i=0,1)$, let $\beta_i=\beta_i'+\beta_i''$. The set
\begin{equation*}\{f\in\PHom_A(\beta_1,\beta_0)\mid f\text{ has a genuine subpresentation in }\PHom_A(\beta_1',\beta_0')\}\end{equation*}
has codimension $\e(\beta_0'-\beta_1',\beta_0''-\beta_1'')$ in $\PHom_A(\beta_1,\beta_0)$.
\end{proposition}

\begin{proof} Since $Z$ is smooth, $\dim Z=\dim\Img p_1+\min_{x\in Z}\{\dim\T_x p_1^{\minus 1}(f)\}$ by a theorem on the generic smoothness \cite[Corollary III.10.7]{Ha}. Combining the dimension formulas in Lemma \ref{tangent} and Lemma \ref{bundle}, we get the codimension:
\begin{align*}
&\quad\ \hom_A(P_1,P_0)-\dim\Img p_1 \\
&=\hom_A(P_1^s,P_0^q)-(\hom_A(P_1^s,P_1^q)+\hom_A(P_0^s,P_0^q))+\min_{x\in Z}\{\dim\Ker(f_0',\minus f_1'')\}\\
&=\min_{x\in Z}\{\dim\Coker(f_0',\minus f_1'')\}.
\end{align*}
This minimal value is exactly $\e(\beta_0'-\beta_1',\beta_0''-\beta_1'')$ by the proceeding remarks.
\end{proof}

\begin{theorem}\label{SPT} The following statements are equivalent: \begin{enumerate}
\item{} A general presentation in $\PHom_A(\beta_1,\beta_0)$ has a genuine subpresentation in $\PHom_A(\beta_1',\beta_0')$.
\item{} $\e(\beta_0'-\beta_1',\beta_0''-\beta_1'')=0$.
\item{} There are $f'\in\Hom_A(P_1',P_0'),f''\in\Hom_A(P_1'',P_0'')$ with $N=\Coker f''$ such that
$\Hom_A(P_0',N)\xrightarrow{f_*'}\Hom_A(P_1',N)$ is surjective.
\end{enumerate}\end{theorem}

\begin{proof}
$(1) \Leftrightarrow (2)$ follows from Proposition \ref{SPP}; $(2) \Leftrightarrow (3)$ follows from Lemma~\ref{DC}.
\end{proof}

For the rest of this section, we suppose that $Q$ is a quiver without oriented cycle and $A$ is the path algebra $kQ$. Recall that any $M\in\Rep_\alpha(kQ)$ admits a canonical presentation \eqref{canproj}:
\begin{equation*}0\to P(\beta_1)=\bigoplus_{a\in Q_1}P_{ha}\otimes M(ta)\xrightarrow{f_M^{can}} P(\beta_0)=\bigoplus_{v\in Q_0}P_v\otimes M(v)\to M\to 0,\end{equation*} where $f_M^{can}$ is given by $e_{ha}\otimes m\mapsto e_{ha}\otimes am-a\otimes m$.
So $(\beta_1,\beta_0)=(D\alpha,\alpha)$, where $D$ is the matrix whose rows and columns are labeled by $Q_0$, with the diagonal entries all zero and the other entries $D_{u,v}$ equal to the number of arrows from $v$ to $u$. Moreover, the cokernel of a general presentation in $\PHom_{kQ}^{can}(\alpha):=\PHom_{kQ}(D\alpha,\alpha)$ corresponds to a general element in $\Rep_\alpha(kQ)$ by Theorem \ref{T:PP2Rep}. As in \cite{S}, we denote by $\ext(\alpha_1,\alpha_2)$ the minimal value of the upper semi-continuous function $\ext_{kQ}(-,-)$ on $\Rep_{\alpha_1}(kQ)\times\Rep_{\alpha_2}(kQ)$.

\begin{corollary}[Schofield] \label{SC} A general representation in $\Rep_{\alpha_1+\alpha_2}(kQ)$ has a subrepresentation in $\Rep_{\alpha_1}(kQ)$ if and only if $\ext(\alpha_1,\alpha_2)=0$.
\end{corollary}

\begin{proof}
$\Rightarrow$. If $M\in\Rep_{\alpha_1+\alpha_2}(kQ)$ is general and let $L\in\Rep_{\alpha_1}(kQ)$ be its subrepresentation. We can splice the canonical presentations of $L$ and $M/L$ together to get a presentation of $M$ in $\PHom_{kQ}^{can}(\alpha_1+\alpha_2)$. We can conclude that a general presentation in $\PHom_{kQ}^{can}(\alpha_1+\alpha_2)$ has a genuine subpresentation in $\PHom_{kQ}^{can}(\alpha_1)$. So by the above corollary and Lemma \ref{DC}, $\ext(\alpha_1,\alpha_2)=0$.

$\Leftarrow$. Suppose that $\ext(\alpha_1,\alpha_2)=0$, then $\e(\alpha_1-D\alpha_1,\alpha_2-D\alpha_2)=0$ by Lemma \ref{DC}, so a general presentation $f$ in $\PHom_{kQ}^{can}(\alpha_1+\alpha_2)$ has a genuine subpresentation $f'$ in $\PHom_{kQ}^{can}(\alpha_1)$.
We can choose $f'$ to be general in $\PHom_{kQ}^{can}(\alpha_1)$, in particular injective. By Corollary \ref{PPRep},$M:=\Coker(f)$ is general in $\Rep_{\alpha_1+\alpha_2}(kQ)$. It follows by the snake lemma that $M$ has a subrepresentation $L=\Coker f'$ in $\Rep_{\alpha_1}(kQ)$.
\end{proof}

\section{Canonical Decomposition of Presentations} \label{S:CD}

Let $\Gr_{\oplus^2}(P)=\{(P',P'')\in\Gr_{P^s}^{P^q}(P)\times \Gr_{P^q}^{P^s}(P)\mid P'\cap P''=0\}$.
This is open in $\Gr_{P^s}^{P^q}(P)\times \Gr_{P^q}^{P^s}(P)$ and hence smooth irreducible.
Let $\Grs_{\oplus^2}=\Gr_{\oplus^2}(P_1)\times\Gr_{\oplus^2}(P_0)$ and we define
\begin{equation*}Z_{\oplus^2}=\{(f,P_1',P_1'',P_0',P_0'')\in\Hom_A(P_1,P_0)\times \Grs_{\oplus^2}\mid f(P_1')\subseteq P_0',f(P_1'')\subseteq P_0''\}.\end{equation*} Let $p_1:Z_{\oplus^2}\to \Hom_A(P_1,P_0)$  be the first factor projection and $p_2:Z_{\oplus^2}\to \Grs_{\oplus^2}$  be the second one. Any point $x=(f,P_1',P_1'',P_0',P_0'')\in Z_{\oplus^2}$ determines up to automorphisms of $P_1',P_0',P_1'',P_0''$,
the following diagram with split rows.
\begin{displaymath} \begin{matrix}
0 & \to & P_0' & \underset{\pi_0'}{\overset{\iota_0'}{\rightleftarrows}} & P_0 & \underset{\iota_0''}{\overset{\pi_0''}{\rightleftarrows}} & P_0'' & \to & 0\\
& & \quad\uparrow_{f'} & & \quad\uparrow_{f} & & \quad\uparrow_{f''}\\
0 & \to & P_1' & \underset{\pi_1'}{\overset{\iota_1'}{\rightleftarrows}} & P_1 & \underset{\iota_1''}{\overset{\pi_1''}{\rightleftarrows}} & P_1'' & \to & 0\\
\end{matrix} \end{displaymath}

\begin{lemma}\
{\em (i)}  $Z_{\oplus^2}$ is a vector bundle over $\Grs_{\oplus^2}$ with fibre $p_2^{\minus 1}(P_1',P_1'',P_0',P_0'')$ isomorphic to $\Hom_A(P_1',P_0')\oplus\Hom_A(P_1'',P_0'')$.

{\em (ii)}  The tangent space $\T_xp_1^{\minus 1}(f)$ is isomorphic to $\Ho(f',f'')\oplus\Ho(f'',f')$.
\end{lemma}

\begin{proof}
(i) Similar to Lemma \ref{bundle}, $Z_{\oplus^2}$ can be realized as the subbundle of the trivial vector bundle $\Hom_A(P_1,P_0)$ on $\Grs_{\oplus^2}$
$$Z_{\oplus^2}=\{(P_1',P_1'',P_0',P_0'',f)\in \Grs_{\oplus^2} \times \Hom_A(P_1,P_0)\mid \pi_0''f\iota_1'=0, \pi_0'f\iota_1''=0\}.$$
Then the fibre $p_2^{\minus 1}(P_1',P_1'',P_0',P_0'')$ is isomorphic to $\Hom_A(P_1',P_0')\oplus\Hom_A(P_1'',P_0'')$.

(ii) For any $(F,\varphi_1',\varphi_1'',\varphi_0',\varphi_0'')\in\T_x(\Hom_A(P_1,P_0)\times \Grs_{\oplus^2})$
\begin{equation*}\cong\Hom_A(P_1,P_0)\oplus\Hom_A(P_1',P_1'')\oplus\Hom_A(P_1',P_1'')\oplus\Hom_A(P_0',P_0'')\oplus\Hom_A(P_0'',P_0'),\end{equation*} as in Lemma \ref{tangent} one can show that $(F,\varphi_1',\varphi_1'',\varphi_0',\varphi_0'')\in\T_xZ_{\oplus^2}$ if and only if $\pi_0''F\iota_1'=\varphi_0' f'-f''\varphi_1'$ and $\pi_0'F\iota_1''=\varphi_0'' f''-f'\varphi_1''$. Hence $\T_xp_1^{\minus 1}(f)\cong\Ho(f',f'')\oplus\Ho(f'',f')$.
\end{proof}

\begin{corollary} \label{CDC} For any $\beta_i',\beta_i''\in\mathbb{N}^{Q_0}$, let $\beta_i=\beta_i'+\beta_i''$. The set
\begin{equation*}\{f\in\PHom_A(\beta_1,\beta_0)\mid f=f'\oplus f''\text{ with }f'\in\PHom_A(\beta_1',\beta_0'),f''\in\PHom_A(\beta_1'',\beta_0'')\}\end{equation*}
has codimension $\e(\beta_0'-\beta_1',\beta_0''-\beta_1'')+\e(\beta_0''-\beta_1'',\beta_0'-\beta_1')$ in $\PHom_A(\beta_1,\beta_0)$.
\end{corollary}

\begin{proof} $Z_{\oplus^2}$ is smooth and irreducible by the first part of the above lemma. We will perform the dimension counting using the two projections of $Z_{\oplus^2}$. From the projection $p_1$, we know from the second part of the above lemma and the generic smoothness that
\begin{equation*}\dim Z_{\oplus^2}=\dim\Img p_1+\min_{x\in Z_{\oplus^2}}\{\dim(\Ho(f',f'')\oplus\Ho(f'',f'))\}.\end{equation*}
While the projection $p_2$ yields that
\begin{equation*}\dim Z_{\oplus^2}=\dim \Grs_{\oplus^2}+\hom_A(P_1^s,P_0^s)+\hom_A(P_1^q,P_0^q).\end{equation*}
So the codimension:\begin{align*}
&\quad\ \! \hom_A(P_1,P_0)-\dim\Img p_1 \\
&=\hom_A(P_1,P_0)-\dim \Grs_{\oplus^2}-(\hom_A(P_1^s,P_0^s)+\hom_A(P_1^q,P_0^q))\\
&\quad\ +\min_{x\in Z_{\oplus^2}}\{\dim(\Ho(f',f'')\oplus\Ho(f'',f'))\}\\
&=\hom_A(P_1^s,P_0^q)-(\hom_A(P_1^s,P_1^q)+\hom_A(P_0^s,P_0^q))+\min_{x\in Z_{\oplus^2}}\{\dim\Ho(f',f'')\}\\
&+\hom_A(P_1^q,P_0^s)-(\hom_A(P_1^q,P_1^s)+\hom_A(P_0^q,P_0^s))+\min_{x\in Z_{\oplus^2}}\{\dim\Ho(f'',f')\}\\
&=\min_{x\in Z_{\oplus^2}}\{\e(f',f'')+\e(f'',f')\}.
\end{align*} This minimal value is exactly $\e(\beta_0'-\beta_1',\beta_0''-\beta_1'')+\e(\beta_0''-\beta_1'',\beta_0'-\beta_1')$.
\end{proof}

Let $l$ be a fixed integer greater than $1$, then the abelian category $\Com^l(\Rep(A))$ of bounded complexes of length $l$ can be viewed as the category $\Rep(A_l)$ for some finite-dimensional algebra $A_l$, which has Krull-Schmidt property \cite[Corollary I.4.8]{ASS}. So its full subcategory consisting of projective presentations is Krull-Schmidt as well. In particular, every presentation $f$ has a unique decomposition $f=f_1\oplus f_2\oplus\cdots\oplus f_s$ with each $f_i$ an indecomposable object in Com$^l(\Rep(A))$. Note that each $f_i$ is a projective presentation.

\begin{definition} $\delta\in\mathbb{Z}^{Q_0}$ is called {\em indecomposable} if a general presentation in $\PHom_A(\delta)$ is indecomposable.
We call $\delta=\delta_1\oplus \delta_2\oplus\cdots\oplus\delta_s$ the {\em canonical decomposition} of $\delta$ if a general element in $\PHom_A(\delta)$ decompose into (indecomposable) ones in each $\PHom_A(\delta_i)$.
\end{definition}

It is easy to see that if a general presentation $f$ decomposes as $f=f_1\oplus f_2\oplus\cdots\oplus f_s$, then each summand $f_i\in\PHom_A(\delta_i)$  can be chosen to be general too, so each $\delta_i$ in the canonical decomposition is indecomposable. Thus Corollary \ref{CDC} can be easily generalized to the following.

\begin{theorem} \label{CDT} $\delta=\delta_1\oplus \delta_2\oplus\cdots\oplus\delta_s$ is the canonical decomposition of $\delta$ if and only if $\delta_1,\cdots,\delta_s$ are indecomposable, and $\e(\delta_i,\delta_j)=0$ for $i\neq j$.
\end{theorem}

\begin{proof}
$\Rightarrow$. We see from the canonical decomposition of $\delta$ that a general presentation in $\PHom_A(\delta)$ has a summand in $\PHom_A(\delta_i+\delta_j)\ (i\neq j)$. But this summand is general in $\PHom_A(\delta_i+\delta_j)$ and it is a direct sum of two presentations in $\PHom_A(\delta_i)$ and $\PHom_A(\delta_j)$. By Corollary \ref{CDC}, $\e(\delta_i,\delta_j)=0$.

$\Leftarrow$. We prove by induction on $s$. $s=1$ is the trivial case. Now assume this is true for $s-1$, then $\delta'=\delta_1\oplus\cdots\oplus\delta_{s-1}$ is the canonical decomposition of $\delta'$ and clearly $\e(\delta',\delta_s)=\e(\delta_s,\delta')=0$. By Corollary \ref{CDC}, a general presentation $f\in\PHom_A(\delta)$ can be decomposed to $f'\in\PHom_A(\delta')$ and $f''\in\PHom_A(\delta_s)$, but $f'$ is general in $\PHom_A(\delta')$, so we can finish the proof by the induction hypothesis.
\end{proof}

The above theorem can be specialized to the case of path algebras in a similar fashion as Corollary \ref{SC}.

\begin{corollary}[Kac] \label{KC} $\alpha=\alpha_1\oplus \alpha_2\oplus\cdots\oplus\alpha_s$
is the canonical decomposition of $\alpha$ if and only if $\ext(\alpha_i,\alpha_j)=0$ for $i\neq j$ and a general representation in $\Rep_{\alpha_i}(kQ)$ is indecomposable for each $\alpha_i$ .
\end{corollary}

Recall from Section \ref{S:PP} that $\E(f,f)$ can be interpreted as the normal space of $f$ in $\Hom_A(P_1,P_0)$.

\begin{definition} A presentation $f$ is called {\em rigid} if $\E(f,f)=0$. An indecomposable $\delta\in\mathbb{Z}^{Q_0}$ is called {\em real} if there is a rigid $f\in\PHom_A(\delta)$; is called {\em tame} if it is not real but $\e(\delta,\delta)=0$; is called {\em wild} if $\e(\delta,\delta)>0$.
\end{definition}

\noindent If $\delta$ is real or tame, then by Theorem \ref{CDT}, $m\delta=\underbrace{\delta\oplus\cdots\oplus\delta}_m$ is the canonical decomposition for any $m\in\mathbb{N}$. In particular, $\delta$ is indivisible. \begin{question}If $\delta$ is wild, are all $m\delta$ wild (in particular indecomposable)? \end{question}

\section{Maximal Rigid Presentations} \label{S:RP}

\begin{definition} A rigid presentation $f$ is called {\em generating} if $K^b(\proj A)$ is generated by $\Ind(f):=$\{nonisomorphic indecomposable direct summands of $f$\}.
\end{definition}

We will first show how to turn a rigid presentation $f$ into a generating presentation.
The procedure is very similar to the classical one given by K. Bongartz.
We denote the object complex with $A$ in degree $0$ simply by $A$, then $\E(A,f)=0$.
Let $e=\e(f,A)=\hom_{K^b(\proj A)}(f[\minus 1],A)$.
We choose a basis of $\E(f,A)$ and take $f^e[\minus 1]\xrightarrow{can} A$ to be the canonical map with respect to this basis.
Using the triangulated structure of $K^b(\proj A)$, we can complete the above map to a triangle
\begin{equation}\label{triangle+}f^+[\minus 1]\to f^e[\minus 1]\xrightarrow{can} A\to f^+.\end{equation}
It is easy to see that  the mapping cone $f^+$ is a complex concentrated in degree $0,\minus 1$, in other words, it is a presentation. Let us verify that $\E(f,f^+)=0$. We denote $\Hom_{K^b(\proj A)}$ simply by $\Hom$. Apply $\Hom(f,-)$ to the triangle \eqref{triangle+}, we get
\begin{equation*}\Hom(f,f^e)\xrightarrow{\partial}\Hom(f,A[1])\to\Hom(f,f^+[1])\to\Hom(f,f^e[1])=0.\end{equation*}
By construction \ref{triangle+}, $\partial$ is surjective so we have that $\E(f,f^+)=\Hom(f,f^+[1])=0$.  Apply $\Hom(-,f)$ and $\Hom(-,f^+)$, we get two exact sequences
\begin{align*}& 0=\Hom(f^e,f[1])\to\Hom(f^+,f[1])\to\Hom(A,f[1])=0,\\
& 0=\Hom(f^e,f^+[1])\to\Hom(f^+,f^+[1])\to\Hom(A,f^+[1])=0.\end{align*}
So we can conclude $\E(f\oplus f^+,f\oplus f^+)=0$, i.e., $f\oplus f^+$ is rigid, and in fact generating by the above triangle. We call $\tilde{f}^+:=f\oplus f^+$ the positive completion of $f$. Similarly, we can construct the negative completion of $f$ by using $A[1]$ instead of $A$. Namely, let $e^{-}=\e(A[1],f)=\hom(A,f)$ and take the triangle
\begin{equation}\label{triangle-}f^-[\minus 1]\to A\xrightarrow{can} f^{e^{-}}\to f^-.\end{equation}
Then $\tilde{f}^-=f\oplus f^-$ is the negative completion of $f$.

Next, we claim that $n(f):=|\Ind(f)|\leqslant r(A):=\rank(K_0(\Rep(A)))$ for rigid $f$. We are going to prove this through classical tilting theory. The following construction will be used to pass rigid presentations to classical tilting modules.

\begin{definition}
The {\em universal regularization} of $A$ with respect to a projective presentation $P_1\xrightarrow{f} P_0$ is an algebra epimorphism $\pi^f:A\twoheadrightarrow A^f$ universal with respect to the property that $A^f\otimes_A P_1 \xrightarrow{A^f\otimes_A f}A^f\otimes_A P_0$ is injective.
\end{definition}

Here is a concrete construction of $\pi^f$. Suppose that $P_0=\bigoplus_{i=1}^m P_{u_i}$ and $P_1=\bigoplus_{j=1}^n P_{v_j}$, where $P_{u_i},P_{v_j}$ are indecomposable projective. Then $P_0,P_1$ are (column) vectors with entries in $P_{u_i},P_{v_j}$, and $f$ can be represented by an $n\times m$ matrix $(a_{ji})$ with $a_{ji}$ a linear combination of paths from $u_i$ to $v_j$. Let $I_0$ be the two-sided ideal in $A$ generated by the entries of all vectors in $\Ker f$ and $A_1=A/I_0$. If $A_1\otimes_A f$ becomes injective, then we take $A^f=A_1$ and $\pi^f$ to be the canonical projection. Otherwise, we repeat the previous step for the projective presentation in $\Rep(A_1): A_1\otimes_A P_1 \xrightarrow{A_1\otimes_A f}A_1\otimes_A P_0$. This procedure must terminate in finitely many, say $s$ steps. We get a sequence of projections:
\begin{equation*}A_0=A\to A_1=A_0/I_0\to\cdots\to A_s=A_{s-1}/I_{s-1}.\end{equation*}
Our desired $\pi^f$ is the composition of these projections. In particular, if $f_v:=P_v\to 0$, then by construction $A^{f_v}=A/\langle e_v\rangle$, where $e_v$ is the trivial path corresponding to the vertex $v$.

\begin{lemma} \label{UR}
If $\E(f,f')=0$ and $M\in\Rep(A)$ is the cokernel of $P_1'\xrightarrow{f'} P_0'$, then $M$ is in fact a representation of $A^f$ $($of projective dimension one$)$, and moreover $\E(A^f\otimes_A f,A^f\otimes_A f')=0$.
\end{lemma}

\begin{proof}
Suppose that $\cdots\to P_2\xrightarrow{g}P_1\xrightarrow{f}P_0$ is a projective resolution, then $g$ is represented by elements in $\Ker f$. Apply $\Hom(-,M)$ to the resolution, we get
\begin{equation*}\Hom_A(P_0,M)\to\Hom_A(P_1,M)\xrightarrow{g^M}\Hom_A(P_2,M)\to\cdots.\end{equation*}
Since $\E(f,f')=0$, $g^M$ must be a zero map by Lemma \ref{DC}. In other words, the representation $M$ satisfies all the relations generated by the entries of all vectors in $\Ker f$. So $M\in\Rep(A_1)$, where $A_1=A/I_0$ as in the construction above. In the mean while, $\E(A_1\otimes_A f,A_1\otimes_A f')$ still equals $0$, so we can complete the proof by induction.
\end{proof}

Now we prove the claim that $n(f)\leqslant r(A)$ for rigid $f$. Suppose that $n(f)>r(A)$ and in addition $P_v[1]\notin\Ind(f)$. Applying the above lemma to $f'=f$, we see by Lemma \ref{DC} that $M=\Coker(f)$ as a classical (partial) tilting $A^f$-module has more than $r(A)=r(A^f)$ nonisomorphic indecomposable summands, which contradicts the classical tilting theory \cite[Corollary VI.4.4]{ASS}. If $\Ind(f)$ contains some $P_v[1]$, apply the above lemma to $f_v$ and $f$, then we can easily reduce to the previous situation for another algebra $A'$ with $r(A')<r(A)$.

\begin{theorem} \label{MRT} For a rigid presentation $f$, the following are equivalent:\begin{enumerate}
\item{} $|\Ind(f)|=\rank(K_0(\Rep(A)))$.
\item{} $f$ is maximal rigid in the sense that $\E(f\oplus f',f\oplus f')\neq 0$ for any indecomposable $f'\notin\Ind(f)$.
\item{} $f$ is generating.
\end{enumerate}\end{theorem}

\begin{proof} We have already proved (1) implies (2). Now assume that $f$ is maximal rigid, then $K^b(\proj A)$ can be generated by $\Ind(f)$, otherwise the (positive or negative) completion of $f$ generates $K^b(\proj A)$ contradicting the maximality of $f$. Hence (2) implies (3). Finally if $f$ is generating, then $n(f)\geqslant r(A)$ because $K_0(K^b(\proj A))$ is isomorphic to $K_0(\Rep(A))$. So (3) implies (1) finishing the proof.
\end{proof}

\begin{remark} The above proposition may be proved entirely in the derived category setting. Our $ad\ hoc$ approach looks simpler but depends on the classical tilting theory.
\end{remark}

\begin{definition}
For any maximal rigid presentation $f$, suppose that $\Ind(f)=\{f_1,f_2,\cdots,f_n\}$ and let $f_{\widehat{k}}=\oplus_{i\neq k}f_i$. If $\Ind(f)=\Ind(\tilde{f}_{\widehat{k}}^+)$ (resp. $=\Ind(\tilde{f}_{\widehat{k}}^-)$), then we say $f_k$ is the $positive$ (resp. $negative$) $complement$ of $f_{\widehat{k}}$, denoted by $f_+$ (resp. $f_-$). We define the $mutation$ of $\Ind(f)$ at $f_k$ to be $\Ind(\tilde{f}_{\widehat{k}}^-)$ (resp. $\Ind(\tilde{f}_{\widehat{k}}^+)$).
\end{definition}

Now suppose that $|\Ind(f)|=n-1$ and $f_+,f_-$ be the positive and negative complements. First we claim that $f_+$ and $f_-$ are always different. Let $H$ be the hyperplane in $K_0(K^b(\proj A))\cong\mathbb{Z}^{|Q_0|}$ spanned by the classes in $\Ind(f)$. Since the classes of $A$ and $A[1]$ lie in two different sides of $H$, we see from the two triangles \eqref{triangle+},\eqref{triangle-} that the classes of $f_+$ and $f_-$ are also separated by $H$.

Let $f_c$ be any complement of $f$ and we claim that $f_c$ must be either $f_+$ or $f_-$.  Otherwise, $f_c$ and one of $\delta(f_+),\delta(f_-)$ must  stay in the same side of $H$. Suppose that $\delta(f_c)$ and $\delta(f_+)$ live together, then the interiors of rational convex cones spanned by the classes in $\Ind(f\oplus f_+)$ and $\Ind(f\oplus f_c)$ must intersect, so by Theorem \ref{CDT} we can find some $\delta$ which possesses two canonical decompositions, one involving $\delta(f_+)$ and the other involving $\delta(f_c)$ but no $\delta(f_+)$. This contradicts the uniqueness of the canonical decomposition.

Applying $\Hom(f_+,-)$ to the triangle $f^-[\minus 1]\to A\xrightarrow{can} f^{e^{-}}\to f^-$, we get
\begin{equation*}0=\Hom(f_+,f^e[1])\to\Hom(f_+,f^-[1])\to\Hom(f_+,A[2])=0.\end{equation*}
So $\E(f_+,f^-)=0$ and hence $\E(f_+,f_-)=0$. Then $\e(f_-,f_+)=d\neq 0$ by Theorem \ref{MRT}. Complete $f_-^d\xrightarrow{can} f_+[1]$ to a triangle $f_+\to f'\to f_-^d\xrightarrow{can} f_+[1]$ and apply $\Hom(f_-,-)$ and $\Hom(-,f_-)$ to it, we get
\begin{gather*}
\Hom(f_-,f_-^d)\xrightarrow{can}\Hom(f_-,f_+[1])\to\Hom(f_-,f'[1])\to\Hom(f_-,f_-^d[1])=0,\\
0=\Hom(f_-^d,f_-[1])\to\Hom(f',f_-[1])\to\Hom(f_+,f_-[1]))=0.
\end{gather*}
So $\E(f_-,f')=\E(f',f_-)=0$. Apply $\Hom(f_+,-)$ to the same triangle and we get
\begin{equation*}0=\Hom(f_+,f_+[1])\to\Hom(f_+,f'[1])\to\Hom(f_+,f_-^d[1])=0,\end{equation*}
so $\E(f_+,f')=0$. Finally apply $\Hom(-,f')$ to the triangle again and we get
\begin{equation*}0=\Hom(f_-^d,f'[1])\to\Hom(f',f'[1])\to\Hom(f_+,f'[1])=0,\end{equation*}
so $\E(f',f')=0$. Hence $\E(f'\oplus f_-,f'\oplus f_-)=0$. Similar argument goes through the triangle $f_+^d\to f''\to f_-\xrightarrow{can} f_+^d[1]$ as well. We summarize as follows.

\begin{proposition} \label{+-}
$f_+,f_-$ are the two and only two complements of $f$. They are related by the triangle $f_+\to f'\to f_-^d\to f_+[1]$ and $f_+^d\to f''\to f_-\to f_+^d[1]$, where $d=\e(f_-,f_+)$. Moreover, both $f'\oplus f_-$ and $f''\oplus f_+$ are rigid and $\E(f_+,f_-)=\E(f_+,f')=\E(f'',f_-)=0$. In particular, $d=1$ if and only if  $f'=f''$ belongs to the subcategory generated by $\Ind(f)$.
\end{proposition}

\begin{question} What is a necessary and sufficient condition for $d=1$?\end{question}

\section{The Simplicial Complexes} \label{S:SC}

Now we can attach to the algebra $A$ an abstract simplicial complex $\mc{S}(A)$ as follows. The set $\mc{S}(A)_p$ of $p$-simplexes consists of all $\{\delta_1,\dots,\delta_p\}$ such that each $\delta_i$ is indecomposable and $\e(\delta_i,\delta_j)=0$ for $i\neq j$. Let $\mc{S}^r(A)$ be the subcomplex of $\mc{S}(A)$ consisting of all the simplexes whose vertexes are all real. When $A=kQ$ is a finite-dimensional path algebra (without relation), $\mc{S}^r(A)$ is the well-known cluster complex by Lemma \ref{DC}.

In general, for any abstract simplicial complex $\mc{K}$, the {\em geometric realization} $|\mc{K}|$ of $\mc{K}$ is defined as follows. Let $\mc{K}_0$ be the set of $0$-simplexes, then $|\mc{K}|$ is the subset of the vector space $\mathbb{R}^{\mc{K}_0}$ consisting of all $x=\sum t_iv_i,t_i\in [0,1],v_i\in\mc{K}_0$ with the property that $\sum t_i=1$ and set of all vertices $v_i$ with nonzero coefficient is a simplex in $\mc{K}$. We topologize $|\mc{K}|$ by giving it the weakest topology with respect to the property that a map $\lambda:|\mc{K}|\to X$ is continuous if and only if it is continuous on every closed simplex, i.e., the set of all $x\in|\mc{K}|$ which are nonnegative linear combinations of the vertices of a simplex in $\mc{K}$.

Let $\Lambda_0:\mc{S}(A)_0\to\mathbb{R}^{|Q_0|}$ be the map assigning each vertex $\delta_i$ to itself, then $\Lambda_0$ can be piecewise-linearly extended to a continuous map $\Lambda:|\mc{S}(A)|\to\mathbb{R}^{|Q_0|}$. We define $\lambda:|\mc{S}(A)|\to S^{|Q_0|-1}$ by $\lambda(\delta)=\Lambda(\delta)/\norm{\delta}$, where $\norm{\cdot}$ is the usual Euclidean norm.

\begin{proposition} The map $\lambda$ is injective. If $\mc{S}^r(A)$ is finite, then $\mc{S}(A)=\mc{S}^r(A)$ and $\lambda$ gives a triangulation of the sphere $S^{|Q_0|-1}$.
\end{proposition}

\begin{proof} Suppose that $\lambda$ is not injective and the fibre over $\varepsilon$ has more than one point, then some integral multiple $k\varepsilon\in\mathbb{Z}^n$ would have two different canonical decompositions by Theorem \ref{CDT}. This contradicts the uniqueness of the canonical decomposition.

When $\mc{S}^r(A)$ is finite, we define $\lambda^r:|\mc{S}^r(A)|\to S^{|Q_0|}$ exactly the same way as $\lambda$. By the construction in Section \ref{S:RP} and Proposition \ref{+-}, the fan corresponding to the simplicial complex $\lambda^r|\mc{S}^r(A)|$ has the two properties of the following lemma. Our claim follows.
\end{proof}

\begin{lemma}
Suppose that
${\mathcal F}\subseteq \R^n$ is a closed fan. Let ${\mathcal C}_i$ be the set of $i$-dimensional simplicial cones in the fan.
Assume that \begin{enumerate}
\item every element $F$ in ${\mathcal C}_i$ with $i<n$ is the face of a simplicial cone in ${\mathcal C}_n$.
\item every element $F$ in ${\mathcal C}_{n-1}$ is the face of two distinct simplicial cones in ${\mathcal C}_n$.
\end{enumerate} Then the fan covers $\R^n$.
\end{lemma}
\begin{proof}
Since every cone $C\in {\mathcal F}_n$ contains a rational point, ${\mathcal F}_n$ is countable. By property (1), ${\mathcal F}_i$ is countable for all $i$.
Suppose that $p\in \R^n$ does not lie in the fan. Let $C\in {\mathcal C}_n$. For every element $F\in {\mathcal C}_{n-2}$, let $\overline{F}$ be the span of $F$ and $p$.
Let $D$ be the union of all $\overline{F}$, with $F\in {\mathcal F}_{n-2}$. Then $D$ is a union of simplicial cones of dimension $\leq n-1$.
Since $D$ has zero measure, it cannot contain $C$. Choose $q\in C\setminus D$. Define $\gamma:[0,1]\to \R^n$ by $\gamma(t)=(1-t)p+tq$. Then $\gamma$ is a continuous path from $p$ to $q$.
We have $\gamma^{-1}({\mathcal F})$ is closed, hence compact. Therefore $\gamma^{-1}({\mathcal F})$ has a smallest element, say $t_0$. Now $\gamma(t_0)$ lies in $C$ for some $C\in {\mathcal C}_i$.
By property (1), we know that $\gamma(t_0)$ lies in $C$ for some $C\in {\mathcal C}_{n}$. Now $\gamma(t_0)$ cannot lie in the interior of $C$ because then $\gamma(t_0-\varepsilon)\in C$ for some small $\varepsilon>0$. Therefore, $\gamma(t_0)$ lies in some $n-1$-dimensional facet $F$ of $C$. If $\gamma(t_0)$ does not lie in the relative interior of $F$, then $\gamma(t_0)$ lies in $F'$ for some $n-2$ dimensional face of $F$. But from $(1-t_0)p+t_0q\in F'$ follows that $q\in \overline{F'}$, which is a contradiction.
Therefore, $\gamma(t_0)$ lies in the relative interior of $F$. Besides $C$ there must be another $n$-dimensional simplicial cone of ${\mathcal F}$ such that $F$ is a facet of $C'$.But then $\gamma(t_0)$ lies in the interior of $C\cup C'$, and $\gamma(t_0-\varepsilon)\in {\mathcal F}$ for some small $\varepsilon>0$. Contradiction.
Therefore, it is not possible to choose $p$ outside the fan, so the fan covers $\R^n$.
\end{proof}

\begin{question}
If the algebra $A$ is of finite type, then $\mc{S}^r(A)$ is finite. Is it true for the converse?
\end{question}

A similar simplicial complex for quivers was studied in \cite{DW}. It was proved in \cite[Theorem 7.1]{HU} that $\mc{S}^r(kQ)$ is always connected. In general, $\mc{S}(A)$ or $\mc{S}^r(A)$ may be disconnected.

\begin{example} Consider the $3$-arrow Yin-Yang quiver $Q$:
$$\yinyangkronecker{u}{v}{_{\times 3}}{_{\times 3}}$$ with three arrows in each direction.
Let $A$ be the algebra of $kQ$ modulo the relations generated by all paths of length 2. A simplicial complex governing the canonical decomposition of general representations of the generalized Kronecker quiver $\Theta:u\xrightarrow{\times 3}v$ was shown in \cite[p. 249]{DW}. It is not hard to see that the simplicial complex is the part of $\mc{S}(A)$ in the 1st and 4th quadrants as shown in the figure and $\mc{S}(A)$ is symmetric about the origin.

\ \centerline{\includegraphics[scale=0.6]{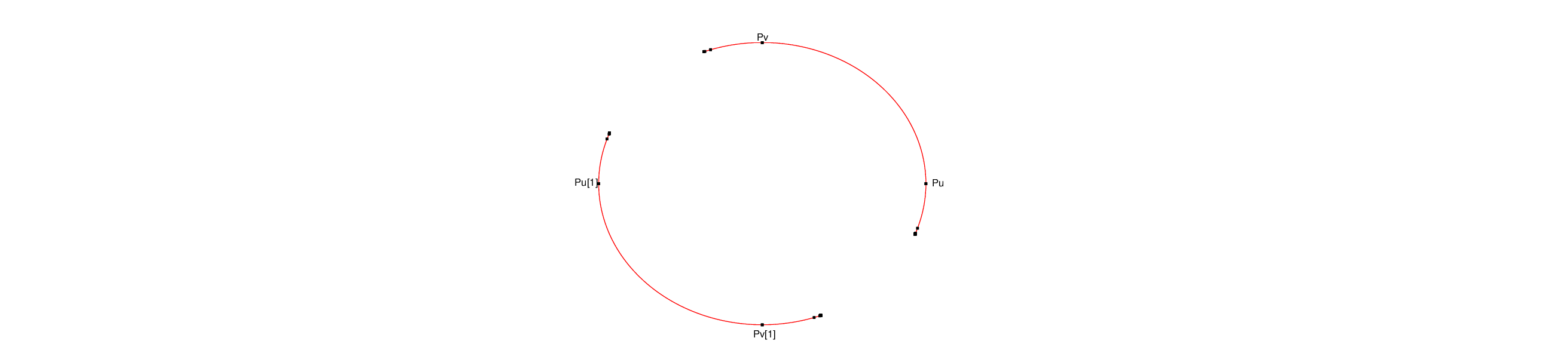}}

The missing part on the circle contains the wild $\delta$-vectors. Note that the algebra $A$ has infinite global dimension.
\end{example}

\begin{question}
Is $\mc{S}(A)$ or $\mc{S}^r(A)$ always connected for connected algebras of finite global dimension?
\end{question}

\begin{example} Consider the quiver $Q$
$$\productthreevertexcomplex{x}{y}{z}{a_1}{a_2}{b_1}{b_2}$$ with relation $b_1a_1=b_2a_2=0$. The algebra $A=kQ/I$ is a string algebra of global dimension 2 with the Cartan matrix $C=\left( \begin{smallmatrix}1&2&2\\0&1&2\\0&0&1\end{smallmatrix} \right)$, i.e., the matrix whose rows are the dimension vectors of the indecomposable projective representations. The general representations of $A$ in irreducible components were studied in \cite{KW}. It is known that there are only three possible cases for the dimension vector $\alpha$ of a generally indecomposable representation. \begin{enumerate}
\item{} $\alpha_-(m,n)=(m-1,m+n-1,n-1),\text{ for } m\geqslant 1,n\geqslant 1;$
\item{} $\alpha_+(m,n)=(m+1,m+n+1,n+1),\text{ for } m\geqslant -1,n\geqslant -1,m+n\geqslant -1;$
\item{} $\alpha_0(m,n)=(m,m+n,n), \text{ for } m\geqslant 0, n\geqslant 0 \text{ with GCD}(m,n)=1$.\end{enumerate}
An easy rank argument can show that all these representations have projective dimension less than 1, so $\delta$-vectors of their minimal presentations are determined by the Cartan matrix. In fact, all the indecomposable $\delta$-vectors are the following \begin{enumerate}
\item{} $\delta_-(m,n)=(m-1,n-m+1,\minus(n+1)),\text{ for } m\geqslant 1,n\geqslant 1,\text{ or }=\minus 1\\ \text{and } (m,n)=(0,\minus 1),(1,0);$
\item{} $\delta_+(m,n)=(m+1,n-m-1,\minus(n-1)),\text{ for } m\geqslant -1,n\geqslant -1,m+n\geqslant -1;$
\item{} $\delta_0(m,n)=(m,n-m,\minus n),\text{ for } m,n\geqslant 0\text{ with GCD}(m,n)=1.$
\end{enumerate}
Except in case (1) when $n=\minus 1$ or $(m,n)=(1,0)$, all the general presentations are injective and the dimensions of the cokernels match with the first list.

Then using Lemma \ref{DC}, it is ready to check that (1),(2) are real $\delta$-vectors and (3) are tame ones. For $m\geqslant 1$ and $n\geqslant 1$, $\tau$ takes the rigid presentation in $\PHom_A(\delta_-(m,n))$ to the rigid presentation in $\PHom_A(\delta_+(m,n))$ and vice versa; and takes a general presentation in $\PHom_A(\delta_0(m,n))$ to a general one in the same space. One can check that the above transformation of the $\delta$-vectors are given by $\minus C^{\T}C^{-1}$. In case (1) when $n=\minus 1$, $\tau$ takes the rigid presentation in $\PHom_A(\delta_-(m,\minus 1))$ to the rigid one in $\PHom_A(\delta_-(m+2,\minus 1))$; in case (2) when $m=\minus 1$, $\tau$ takes the rigid presentation in $\PHom_A(\delta_+(\minus 1,n))$ to the rigid one in $\PHom_A(\delta_+(\minus 1,n-2))$; in case (3) when $mn=0$, $\tau$ takes a general presentation in $\PHom_A(\delta_0(0,n))$ or $\PHom_A(\delta_0(m,0))$ to a general one in the same space. There is one singular case when the rigid presentation is $P_z[1]$. $\tau P_z[1]=I_z$ has the minimal presentation $P_z\oplus P_y^2\to P_x^2$, but this presentation is not general in the corresponding space $\PHom_A((2,\minus 2,\minus 1))$. In fact, its $\delta$-vector decomposes as $(2,\minus 2,\minus 1)=2\cdot(1,\minus 1,0)\oplus(0,0,\minus 1)$. We will see in the next section that this can never happen for the Jacobian algebra of a quiver with potential.

The stereographic projection of $\lambda(\mc{S}(A))$ from the point $(\frac{1}{\sqrt{2}},0,\minus\frac{1}{\sqrt{2}})$ looks like
\begin{center}
\includegraphics[width=5in]{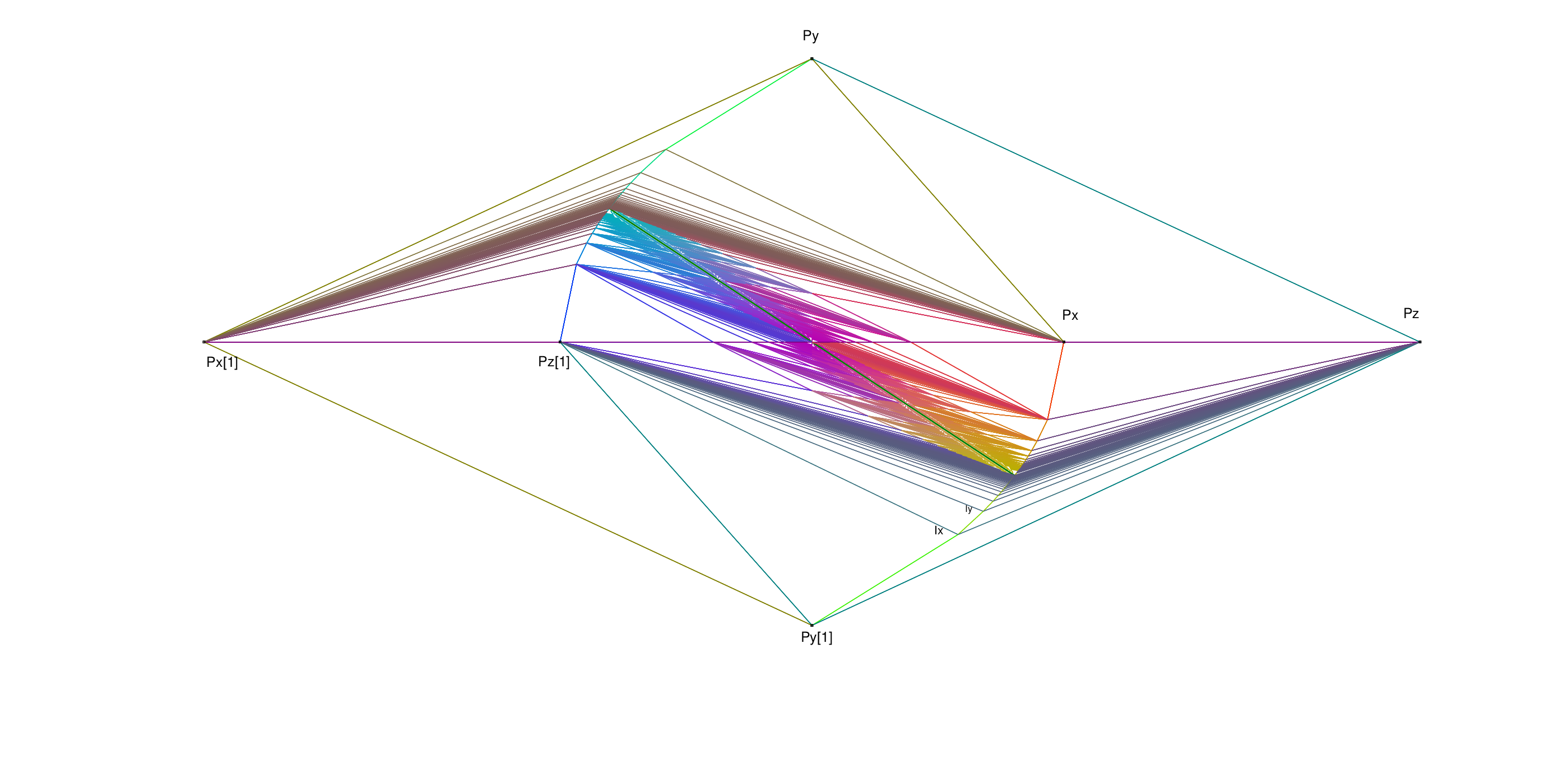}
\end{center}

The dark green river in the middle consists of all the tame $\delta$-vectors. The upper (resp. lower) light green stream is formed by real $\delta$-vectors not involving $P_x$ (resp. $P_z$). Roughly speaking, $\tau$ fixes the river, makes the stream flow, and takes the rest to the other bank. We see that both $\lambda(\mc{S}(A))$ and $\lambda(\mc{S}^r(A))$ are connected and the latter is contractible. For a more detailed account of this example see \cite[5.2]{F}.
\end{example}

\section{Application to Quivers with Potentials} \label{sec:QP}

Recall that a {\em decorated representation} of $A$ is a pair $\mc{M}=(M,V)$, where $M\in\Rep(A)$ is the positive part of $\mc{M}$ and $V\in\Rep(R)$ is the negative part of $\mc{M}$. $\mc{M}$ is called positive (resp. negative) if $V=0$ (resp. $M=0$). Let $\mc{R}ep(A)$ be the set of decorated representations of $A$ up to isomorphism. There is a bijection between $\mc{R}ep(A)$ and $K^2(\proj A)$ mapping any positive representation $M$ to its minimal presentation in $\Rep(A)$, and simple negative representation $\mc{S}_v^-$ to $P_v\to 0$. So the AR-transformation $\tau$ and its inverse are defined in $\mc{R}ep(A)$. They clearly commute with the trivial dual $\mc{M}^*:=\Hom(\mc{M},k)$. For any decorated representation $\mc{M}$, we will write $f(\mc{M})$ for its image in $K^2(\proj A)$, and denote its $\delta$-vector and $i$-th Betti vector by $\delta(\mc{M})$ and $\beta_i(\mc{M})$. Then we have that $\delta(\mc{M})=\beta_0(\mc{M})-\beta_1(\mc{M}),\beta_0(\mc{M})=\beta_0(M)$, and $\beta_1(\mc{M})=\beta_1(M)+\dv{V}$. Note that $\beta_i(M)$ agree with the classical Betti vectors for representations. Moreover, it is clear that $\beta_0(\tau\mc{M}^*)=\beta_1(\mc{M})$ and $\beta_1(\tau\mc{M}^*)=\beta_0(\mc{M})$.

We follow \cite{DWZ2} and define
\begin{definition}
$\E_A^{\prj}(\mc{M},\mc{N})=\E(f(\mc{M}),f(\mc{N}))$ and $\E_A^{\inj}(\mc{M},\mc{N})=\E_{A^{\op}}^{\prj}(\mc{N}^*,\mc{M}^*)$. $\mc{M}$ is called rigid if $f(\mc{M})$ is. We denote $\Hom_A^+(\mc{M},\mc{N}):=\Hom_A(M,N)$.
\end{definition}

From now on, we fix a quiver $Q$ without oriented 2-cycles and a potential $S$ on $Q$. Recall that a {\em potential} $S$ is an element in $\widehat{kQ}/[\widehat{kQ},\widehat{kQ}]$, where $\widehat{kQ}$ is the completion of the path algebra $kQ$. The vector space $\widehat{kQ}/[\widehat{kQ},\widehat{kQ}]$ has as basis the set of oriented cycles up to cyclic permutation. For each arrow $a\in Q_1$, the {\em cyclic derivative} $\partial_a$ on $\widehat{kQ}$ is defined to be $$\partial_a(a_1\cdots a_s)=\sum_{k=1}^{d}a^*(a_k)a_{k+1}\cdots a_sa_1\cdots a_{k-1}.$$ For each potential $S$, its {\em Jacobian ideal} $\partial S$ is the (two-sided) ideal in $\widehat{kQ}$ generated by all $\partial_aS$. Let $A=\widehat{kQ}/\partial S$ be the {\em Jacobian algebra}, and assume it is finite dimensional, in which case completion is unnecessary.
The key notion in \cite{DWZ1} is the definition of mutation $\mu_v$ of the Jacobian algebra $A$ and its decorated representations $\mc{M}$ at some vertex $v\in Q_0$. In fact, the mutation is defined at the level of quivers with potentials, but we do not need this. We refer the readers to that paper for details. In the general context of quivers with relations, the maps $\alpha$ and $\gamma$ involved in the definition of the mutation are exactly $\varphi$ and $\psi$ in \eqref{alphagamma}. So the h-vector and g-vector of $\mc{M}$ defined in \cite[(3.2) and (1.13)]{DWZ2} are nothing but $\minus\beta_0(\mc{M}^*)$ and $\minus\delta(\mc{M}^*)$. One important feature of the mutation is that it is involutive. Moreover, one can check directly from the definition that it commutes with the trivial dual.

\begin{lemma}\label{QPL1} {\em \cite[Lemma 5.2, Proposition 6.1, and Theorem 7.1]{DWZ2}}\begin{enumerate}
\item{} $\beta_0(\mu_v\mc{M})_v=\beta_1(\mc{M})_v$, so $\beta_1(\mu_v\mc{M})_v=\beta_0(\mc{M})_v$ and $\delta(\mu_v\mc{M})_v=\minus\delta(\mc{M})_v$.
\item{} $\hom_{\mu_vA}^+(\mu_v\mc{M},\mu_v\mc{N})-\hom_A^+(\mc{M},\mc{N})=\beta_1(\mc{M})_v\beta_1(\mc{N}^*)_v-\beta_0(\mc{M})_v\beta_0(\mc{N}^*)_v$.
\item{} $\e_{\mu_vA}^{\prj}(\mu_v\mc{M},\mu_v\mc{N})-\e_A^{\prj}(\mc{M},\mc{N})=\beta_0(\mc{M})_v\beta_1(\mc{N})_v-\beta_1(\mc{M})_v\beta_0(\mc{N})_v$, and\\ $\e_{\mu_vA}^{\inj}(\mu_v\mc{M},\mu_v\mc{N})-\e_A^{\inj}(\mc{M},\mc{N})=\beta_1(\mc{M}^*)_v\beta_0(\mc{N}^*)_v-\beta_0(\mc{M}^*)_v\beta_1(\mc{N}^*)_v$.
\end{enumerate}\end{lemma}

\begin{corollary} $\e_A^{\prj}(\mc{M},\mc{M})$ and $\e_A^{\inj}(\mc{M},\mc{M})$ are mutation invariant, so the simplicial complexes $\mc{S}^r(A)$ and $\mc{S}^r(\mu_vA)$ are isomorphic.
\end{corollary}

\begin{question} Does $\mu_v$ send a general presentation in $\PHom_A(\delta)$ to a general presentation in $\PHom_A(\delta')$ for some $\delta'$? \end{question}
\noindent By the proof of \cite[Theorem 1.7]{DWZ2} and \cite[Lemma 5.2]{DWZ2}, in fact we know that $\delta'$ should be given by \cite[(1.3)]{DWZ2}.

\begin{lemma} \label{QPL2} {\em \cite[Corollary 10.8 and Proposition 7.3]{DWZ2}} \begin{enumerate}
\item{} $\E_A^{\prj}(\mc{M},\mc{N})=\Hom_A^+(\mc{N},\tau\mc{M})^*\text{ and }\E_A^{\inj}(\mc{M},\mc{N})=\Hom_A^+(\tau^{\minus 1}\mc{N},\mc{M})^*.$
\item{} $\E_A^{\prj}(\mc{M},\mc{M})=\E_A^{\inj}(\mc{M},\mc{M})$.
\end{enumerate}\end{lemma}

\begin{corollary} \label{SAL}$\E_A^{\prj}(\mc{M},\mc{M})=\E_A^{\prj}(\tau\mc{M},\tau\mc{M})$. In particular, $\tau$ induces a simplicial automorphism on $\mc{S}^r(A)$.
\end{corollary}

\begin{proof}
$\E_A^{\prj}(\mc{M},\mc{M})=\Hom_A^+(\mc{M},\tau\mc{M})^*=\Hom_A^+(\tau^{\minus 1}(\tau\mc{M}),\tau\mc{M})^*=\E_A^{\inj}(\tau\mc{M},\tau\mc{M})\\=\E_A^{\prj}(\tau\mc{M},\tau\mc{M})$. Now if $\mc{M}$ is rigid, then so is $\tau\mc{M}$. Since $\tau$ is a bijection on the set of indecomposable objects, it induces a simplicial automorphism on $\mc{S}^r(A)$.
\end{proof}

\begin{question}Is it true that $\mc{S}(A)$ and $\mc{S}(\mu_vA)$ are isomorphic and $\tau$ induces a simplicial automorphism on $\mc{S}(A)$?\end{question}

The following was shown in \cite{B}.

\begin{lemma} For any finite-dimensional algebra $A$, two representations $M,M'$ of $A$ are isomorphic if and only if $\hom_A(N,M)=\hom_A(N,M')$ for any $N\in\Rep(A)$. There is also a dual statement.
\end{lemma}

\begin{corollary} The following are equivalent: \begin{enumerate}
\item{} Two decorated representations $\mc{M}$ and $\mc{M'}$ are isomorphic.
\item{} For any $\mc{N}\in\mc{R}ep(A)$, $\hom_A^+(\mc{N},\mc{M})=\hom_A^+(\mc{N},\mc{M'})$ and $\e_A^{\inj}(\mc{N},\mc{M})=\e_A^{\inj}(\mc{N},\mc{M'})$.
\item[(2$^*$)] For any $\mc{N}\in\mc{R}ep(A)$, $\hom_A^+(\mc{M},\mc{N})=\hom_A^+(\mc{M'},\mc{N})$ and $\e_A^{\prj}(\mc{M},\mc{N})=\e_A^{\prj}(\mc{M'},\mc{N})$.
\end{enumerate}\end{corollary}

\begin{proof} Since ($2^*$) is the dual of (2), we will prove (2) implies (1) only. By the above lemma, $\mc{M}$ and $\mc{M'}$ have the same positive part. By Lemma \ref{QPL2} and the dual of the above lemma, $\tau^{\minus 1}\mc{M}$ and $\tau^{\minus 1}\mc{M'}$ have the same positive part, then by the definition of $\tau$, $\mc{M}$ and $\mc{M'}$ must have the same negative part.
\end{proof}

\begin{proposition} The AR-transformation $\tau$ commutes with the mutation $\mu_v$ at any vertex $v$.
\end{proposition}

\begin{proof} We first verify that $\hom_A^+(\mc{N},\tau\mu_v\mc{M})=\hom_A^+(\mc{N},\mu_v\tau\mc{M})$ for any $\mc{N}\in\mc{R}ep(A)$. Applying Lemma \ref{QPL1} and \ref{QPL2}, we get the following equation \begin{align*}
\hom_A^+(\mc{N},\tau\mu_v\mc{M}) &=\e_A^{\prj}(\mu_v\mc{M},\mc{N})\\
&=\e_{\mu_vA}^{\prj}(\mc{M},\mu_v\mc{N})+\beta_0(\mc{M})_v\beta_1(\mu_v\mc{N})_v-\beta_1(\mc{M})_v\beta_0(\mu_v\mc{N})_v\\
&=\hom_{\mu_vA}^+(\mu_v\mc{N},\tau\mc{M})+\beta_1(\tau\mc{M}^*)_v\beta_1(\mu_v\mc{N})_v-\beta_0(\tau\mc{M}^*)_v\beta_0(\mu_v\mc{N})_v\\
&=\hom_A^+(\mc{N},\mu_v\tau\mc{M}).\end{align*}
We remain to verify that $\e_A^{\inj}(\mc{N},\tau\mu_v\mc{M})=\e_A^{\inj}(\mc{N},\mu_v\tau\mc{M})$ for any $\mc{N}\in\mc{R}ep(A)$. Applying Lemma \ref{QPL1} and \ref{QPL2}, we get the following equation \begin{align*}
\e_A^{\inj}(\mc{N},\tau\mu_v\mc{M}) &=\hom_A^+(\mu_v\mc{M},\mc{N})\\
&=\hom_{\mu_vA}^+(\mc{M},\mu_v\mc{N})+\beta_1(\mc{M})_v\beta_1(\mu_v\mc{N}^*)_v-\beta_0(\mc{M})_v\beta_0(\mu_v\mc{N}^*)_v\\
&=\e_{\mu_vA}^{\inj}(\mu_v\mc{N},\tau\mc{M})+\beta_0(\tau\mc{M}^*)_v\beta_1(\mu_v\mc{N}^*)_v-\beta_1(\tau\mc{M}^*)_v\beta_1(\mu_v\mc{N}^*)_v\\
&=\e_A^{\inj}(\mc{N},\mu_v\tau\mc{M}).\end{align*}
\end{proof}

\begin{example} Consider the quiver $Q$
$$\threevertextamepotential{v}{w}{u}{b}{c}{d}{a}$$
with the potential $S=cba$. Note that its Jacobian algebra $A$ is of tame type and has infinite global dimension. If we perform the mutation at the vertex $v$, then we get the quiver $Q'$
$$\threevertextame{v}{w}{u}{b^*}{d}{a^*}$$
with zero potential. A similar simplicial complex as $\mc{S}(kQ')$ was studied in \cite{DW}. Note that $(\minus 1,0,1)$ (resp. $(1,0,\minus 1)$) is the only tame $\delta$-vector of $kQ'$ (resp. $A$), and the mutation takes a general presentation in one space to the other, so $\mc{S}(A)=\mc{S}(kQ')$. Their stereographic projections look like

\begin{center}
\includegraphics[width=5in]{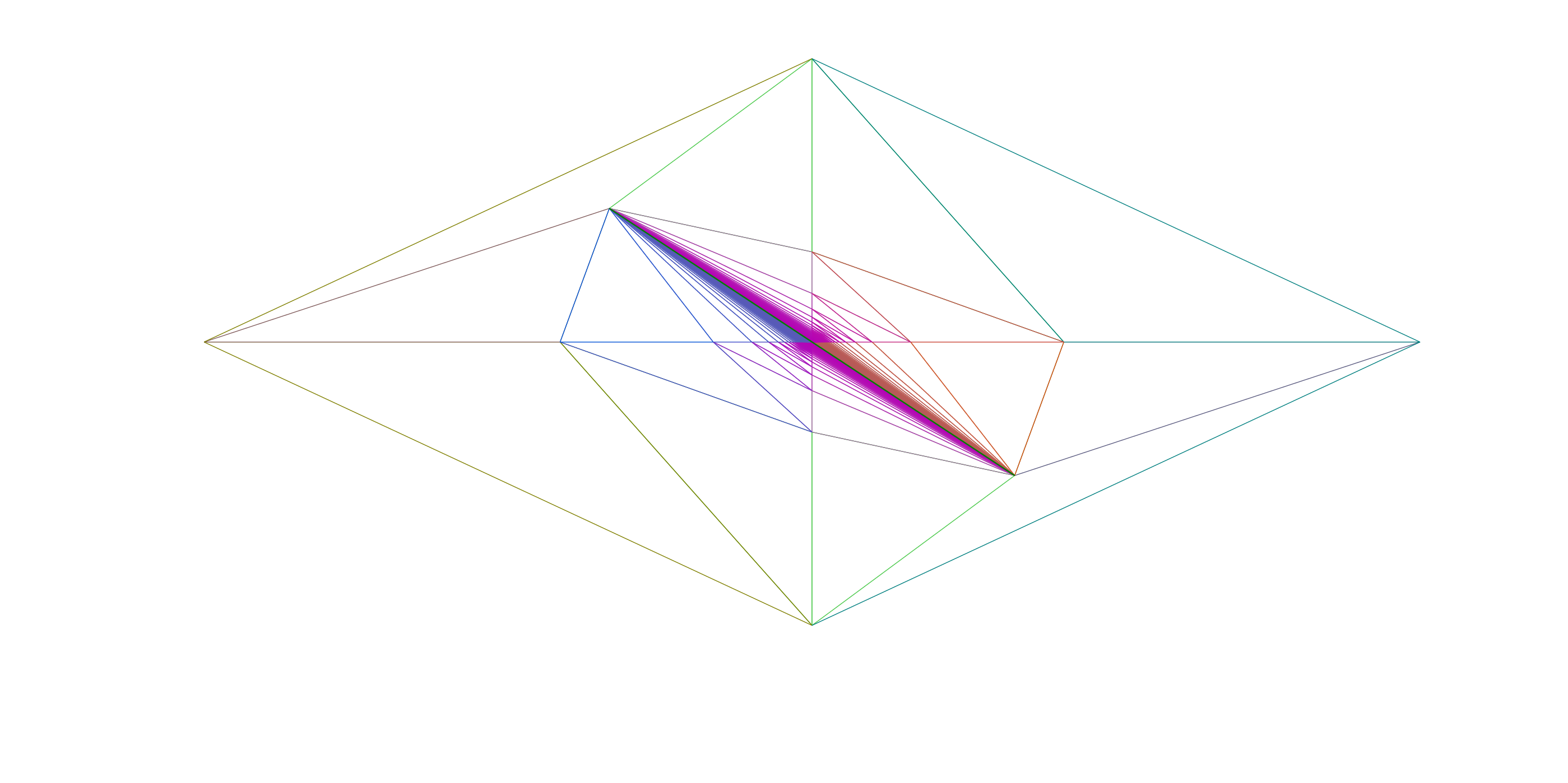}
\end{center}

\end{example}

\section*{Acknowledgement}
The second author wants to thank the first author Harm Derksen for his guidance, continuous support, great patience, and especially high tolerance to one's stupidity.

\bibliographystyle{amsplain}




\end{document}